\documentclass[11pt,reqno]{amsart}
\usepackage{amsmath, amsthm,amssymb,bbm,enumerate,mathrsfs,mathtools}
\usepackage[margin = 2.3cm]{geometry}
\usepackage{xcolor,verbatim}
\usepackage{hyperref}
\usepackage{makecell}
\usepackage{framed}
\usepackage{setspace}
\usepackage{cleveref}
\usepackage{caption}

\usepackage{enumitem}
\setlist[enumerate]{itemsep=3pt,topsep=3pt}
\setlist[enumerate,1]{label=\rm{(\roman*)}}

\usepackage{soul}

\usepackage[square,sort,comma,numbers]{natbib} 
\makeatletter
\renewcommand\subsection{\@startsection{subsection}{2}%
  \z@{-.5\linespacing\@plus-.7\linespacing}{.5\linespacing}%
  {\normalfont\bfseries}}
\makeatother

\allowdisplaybreaks

\date{}

\theoremstyle{plain}

\newtheorem*{thm*}{Theorem}
\newtheorem{thm}{Theorem}
\Crefname{thm}{Theorem}{Theorems}
\numberwithin{thm}{section}

\newtheorem*{lem*}{Lemma}
\newtheorem{lem}[thm]{Lemma}
\Crefname{lem}{Lemma}{Lemmas}

\newtheorem*{claim*}{Claim}
\newtheorem{claim}[thm]{Claim}
\crefname{claim}{Claim}{Claims}
\Crefname{claim}{Claim}{Claims}

\newtheorem{prop}[thm]{Proposition}
\Crefname{prop}{Proposition}{Propositions}

\newtheorem{cor}[thm]{Corollary}
\crefname{cor}{Corollary}{Corollaries}

\newtheorem{conj}[thm]{Conjecture}
\crefname{conj}{Conjecture}{Conjectures}

\Crefname{qn}{Question}{Questions}

\newtheorem{obs}[thm]{Observation}
\Crefname{obs}{Observation}{Observations}

\newtheorem{ex}[thm]{Example}
\Crefname{ex}{Example}{Examples}

\theoremstyle{definition}

\Crefname{prob}{Problem}{Problems}

\Crefname{defn}{Definition}{Definitions}

\theoremstyle{remark}
\newtheorem{rem}[thm]{Remark}
\Crefname{rem}{Remark}{Remarks}

\usepackage{xpatch}
\xpatchcmd{\proof}{\itshape}{\normalfont\proofnamefont}{}{}
\newcommand{\proofnamefont}{}
\renewcommand{\proofnamefont}{\bfseries}

\numberwithin{equation}{section}

\newcommand{\N}{\mathbb{N}}
\newcommand{\F}{\mathcal{F}}
\newcommand{\LL}{\mathcal{L}}
\newcommand{\C}{\mathcal{C}}

\newcommand{\eps}{\varepsilon}

\newcommand{\col}[2]{\mathcal{C}(#1,#2)} 
\newcommand{\lex}[3]{\mathcal{L}(#1,#2,#3)}
\DeclareMathOperator{\lexx}{lex}

\newcommand{\me}[1]{e(#1)}
\newcommand{\ds}[1]{P_2(#1)} 

\newcommand{\floor}[1]{\lfloor #1 \rfloor}
\newcommand{\ceil}[1]{\lceil #1 \rceil}

\begin{document}

\title[Hypergraph Lagrangians I: the Frankl-F\"uredi conjecture is false]{Hypergraph Lagrangians I: \\ \vspace{.1cm} the Frankl-F\"{u}redi conjecture is false}
\author{Vytautas Gruslys \and Shoham Letzter \and Natasha Morrison}

\address{Department of Pure Mathematics and Mathematical Statistics, University of Cambridge, Wilberforce Road, Cambridge CB3 0WB, UK.}\email{v.gruslys|s.letzter|morrison@dpmms.cam.ac.uk.}

\thanks{The second author was supported by Dr.~Max
        R\"ossler, the Walter Haefner Foundation and by the ETH Zurich Foundation. The third author is supported by a research fellowship from Sidney Sussex College, Cambridge.}

\begin{abstract}
	An old and well-known conjecture of Frankl and F\"{u}redi states that the Lagrangian of an $r$-uniform hypergraph with $m$ edges is maximised by an initial segment of colex. In this paper we disprove this conjecture by finding an infinite family of counterexamples for all $r \ge 4$. We also show that, for sufficiently large $t \in \N$, the conjecture is true in the range $\binom{t}{r} \le m \le \binom{t+1}{r} - \binom{t-1}{r-2}$.
\end{abstract}

\maketitle

\section{Introduction}

	Maximising the value of a constrained multilinear function is a problem that emerges naturally in many fields of mathematics. The Lagrangian of a hypergraph, introduced in 1965 by Motzkin and Strauss~\cite{MotStr}, is such a function that has featured prominently in extremal combinatorics: notably with relation to the notorious problem of determining hypergraph Tur\'{a}n densities (see the excellent survey by Keevash~\cite{Kee} as well as \cite{HefKee,BraIrwJia,BenNorYep}).

	Let us begin with some definitions. For a set $X$, let $X^{(r)}$ be the family of all subsets of $X$ of size $r$, i.e.\ $X^{(r)} := \{X \subseteq \N: |X| = r\}$. Given a family $G \subseteq \N^{(r)}$, we think of it as an $r$-uniform hypergraph (or, in short, an $r$-graph) whose edges are the members of $G$. The Lagrangian of a finite $r$-graph $G$ is the maximum value taken by the natural multilinear function corresponding to $G$ over the standard simplex. More formally, we define the \emph{Lagrangian} of $G$ as
	\begin{equation}\label{def:lag}
		\lambda(G) := \max \Bigg\{\sum_{e \in E(G)}\prod_{x \in e}w(x)\,:\, w(x) \ge 0 \text{ for all } x \in \N, \text{ and } \sum_{x \in \N}w(i) = 1\Bigg\}.
	\end{equation}

	Observe that, by compactness, $\lambda(G)$ exists for all finite $G$. For $r=2$, it is a simple exercise to show that $\lambda(G)$ is achieved by equally distributing the weight over the vertices of a largest clique in $G$ (and setting $w(x) = 0$ for all other vertices $x$). However, it is a difficult and well-studied problem to calculate the Lagrangian for a general $r$-graph when $r \ge 3$.     

	In this paper we are interested in determining the maximum value of the Lagrangian over all $r$-graphs with a given number of edges. 	The \emph{colexicographic}, or \emph{colex}, order on $\N^{(r)}$ is the ordering in which $A < B$ if and only if $\sum_{i \in A}2^i < \sum_{i \in B}2^i$. Let $\col{m}{r}$ be the family that consists of the first $m$ sets in the colexicographic order on $\N^{(r)}$. A well-known conjecture of Frankl and F\"{u}redi~\cite{FraFur} from 1989 states that the Lagrangian of an $r$-graph with $m$ edges is maximised by $\col{m}{r}$.  
	
	\begin{conj}[Frankl and F\"{u}redi~\cite{FraFur}] \label{conj:frankl-furedi}
	\label{conj:TheConj}
		Let $G$ be an $r$-graph with $m$ edges. Then 
			$$\lambda(G) \le \lambda(\col{m}{r}).$$
	\end{conj}

	Although initially driven by the possible applications to hypergraph Tur\'{a}n problems, this conjecture has, perhaps due to its beautiful and natural statement, inspired a significant line of research over the past thirty years. Before stating our results, we will summarise the previous work in order to place our theorems in context.

	Given $m$, let $t \in \N$ be such that $\binom{t}{r} \le m < \binom{t+1}{r}$.  
	The majority of results in support of \Cref{conj:frankl-furedi} have focussed on the range where $\binom{t}{r} \le m \le \binom{t+1}{r} - \binom{t-1}{r-2}$, for some $t \in \N$ (where $t$ is often taken to be sufficiently large); we refer to this collection of intervals as the \emph{principal range}. It is natural to consider the principal range, as its $t$-th interval is the set of values $m$ for which the colex graph of size $m$ has the same Lagrangian as the clique $[t]^{(r)}$ (see \Cref{rem:principal-range}). An interesting special case of \Cref{conj:TheConj} is when $m = \binom{t}{r}$ and thus the colex graph of size $m$ is a clique; we (following Tyomkyn~\cite{Tyo}) refer to it as the \emph{principal case}. 

	Motzkin and Strauss~\cite{MotStr} proved \Cref{conj:frankl-furedi} in the case $r=2$. For $r=3$, Talbot~\cite{Tal} showed that it holds whenever $\binom{t}{3} \le m \le \binom{t+1}{3} - \binom{t-1}{1} - t$. Subsequent improvements in the principal range when $r=3$ have been made by Tang, Peng, Zhang and Zhao~\cite{TanPen,TanPen2}, Tyomkyn~\cite{Tyo}, and Lei, Lu and Peng~\cite{LLP}
	who respectively showed that \Cref{conj:frankl-furedi} holds for all values in the $t$-th interval of the principal range except the largest $\frac{t}{2}$, $O(t^{3/4})$ and $O(t^{2/3})$.
	
	For $r \ge 4$, the first progress on \Cref{conj:frankl-furedi} was made by Tyomkyn~\cite{Tyo}, who proved the conjecture in the principal case, i.e.\ when $m = \binom{t}{r}$, for all $r$ and all large $t \in \N$. More recently, Nikiforov~\cite{Nik} reproved the principal case via different, analytic methods for all $r$ and large $t$ (and for $r \le 5$ and all $t$). This was subsequently improved by Lu~\cite{lu} who settled the conjecture in the principal case for all $r$ and $t$.

	In fact, Tyomkyn's impressive work \cite{Tyo} proved the conjecture not only in the principal case but also for almost all values of $m$ in the principal range. In particular, he proved the conjecture for all but the largest $O(t^{r-2})$ values of $m$ in the $t$-th interval of this range, for every $t \in \N$. This was improved recently by Lei and Lu~\cite{LeiLu} who proved the conjecture for all but the largest $O(t^{r-7/3})$ values of $m$ in this interval. 

	Outside of the principal range, all that is known is that the conjecture holds for $r = 3$ and $m = \binom{t}{3} - a$, where $a \in \{1,\ldots,4\}$: The case $a \in \{1,2\}$ was proved by Talbot~\cite{Tal}, and $a \in \{3,4\}$ by Tang, Peng, Zhang and Zhao~\cite{TanPen2}.

	Our first main result shows that \Cref{conj:frankl-furedi} holds for all sufficiently large $m$ in the principal range\footnote{For convenience, in what follows we switch from $t$ to $t-1$.}.
	
	\begin{thm}\label{thm:main1}
		Let $r \ge 3$, let $t \in \N$ be sufficiently large, and let $m$ satisfy $\binom{t-1}{r} \le m \le \binom{t}{r} - \binom{t-2}{r-2}$. Then, for every $r$-graph $G$ with $m$ edges, 
		\[
			\lambda(G) \le \lambda(\col{m}{r}).
		\]
	\end{thm}
	
	Observe that, except for small values of $t$ and for the case $r = 3$ and $m = \binom{t}{3} - a$, where $a \in [4]$, \Cref{thm:main1} contains and improves upon all results stated above concerning \Cref{conj:frankl-furedi} for $r \ge 3$. We also remark that the proof of \Cref{thm:main1} is combinatorial in nature, unlike the proofs given by Lu~\cite{lu} and Nikiforov~\cite{Nik} for the principal case.
	
	Our second main result shows that, for each $r \ge 4$, there exists an infinite family of counterexamples to \Cref{conj:frankl-furedi}.
	
	\begin{thm}\label{thm:main2}
		Let $r \ge 4$, then there exists a constant $\alpha_r > 0$ such that the following holds. Let $t \in \N$ be sufficiently large and let $m:= \binom{t}{r} - \binom{t-2}{r-2} + s$, where $r \le s \le \alpha_r \cdot \binom{t-2}{r-2}$. Then there exists an $r$-graph $G$ with $m$ edges for which
		\[
			\lambda(G) > \lambda(\col{m}{r}).
		\]
	\end{thm}

	In fact, \Cref{thm:main2} is a special case of a much stronger result (\Cref{thm:counter-exs}), that relates the problem of maximising the Lagrangian with the problem of maximising the sum of degrees squared.
	
	In light of \Cref{thm:main2}, it remains to determine whether \Cref{conj:TheConj} is true in the case $r=3$. In a subsequent paper~\cite{Us2}, building upon the groundwork established here, we show that this is indeed the case. 
	
	\begin{thm}[\cite{Us2}]\label{thm:r3}
		Let $m$ be sufficiently large. Then, for any $3$-graph $G$ with $m$ edges, 
		\[
			\lambda(G) \le \lambda(\col{m}{3}).
		\]
	\end{thm}
	
	The first step in the proofs of \Cref{thm:main1,thm:r3} is to show that given $m$ with $\binom{t-1}{r} \le m \le \binom{t}{r}$, there exists a maximiser of the Lagrangian among $r$-graphs with $m$ edges, which is supported on $t$ vertices (see \Cref{thm:t-vs}). Such a result was simultaneously and independently\footnote{We uploaded to arXiv a preliminary version of this paper (see \href{https://arxiv.org/abs/1807.00793v2}{\textcolor{blue}{arXiv1807.00793v2}}), which also included the content of our second paper \cite{Us2}, shortly after Lei and Lu's paper appeared.} proved by Lei and Lu~\cite{LeiLu}. An analogous result for $r=3$ was proved by Talbot~\cite{Tal} and was used in \cite{Tal,TanPen,Tyo,LLP} to prove that \Cref{conj:frankl-furedi} holds in certain ranges. 
	
	Our proofs use some standard results concerning Lagrangians. As they are straightforward, we include them in the next section (\Cref{sec:prelims}), in order to keep this paper complete and self-contained. In \Cref{sec:supp} we prove \Cref{thm:t-vs}, which restricts the support of a maximiser. In \Cref{sec:a-big} we prove \Cref{thm:main1} and in \Cref{sec:counter} we prove \Cref{thm:main2}. We conclude the paper in \Cref{sec:conclusion} with some closing remarks.

\section{Proof sketches} \label{sec:sketch}
	
	In this section we sketch the proofs of our main results. 	

	\subsection{Bounding the support of $G$}
		As an important step towards the proof of \Cref{thm:main1} -- that the Frankl-F\"uredi conjecture holds for all large $m$ in the principal range -- we prove that for any $m$ with $\binom{t-1}{r} \le m \le \binom{t}{r}$ (and large $t \in \N$), there is a maximiser of the Lagrangian among $r$-graphs with $m$ edges which is supported on $t$ vertices (see \Cref{thm:t-vs}). 
		
		To prove this we first prove a few simple facts about such maximisers $G$ and the corresponding weight function $w$: for example, we show that every vertex in $G$ has weight $O(1/t)$, we conclude that the degree of each vertex is large (namely, $\Omega(t^{r-1})$, and ignoring vertices with weight $0$), from which it follows that $G$ has $O(t)$ vertices (see \Cref{prop:bigO}).
		
		Our next step, which is the key step in the proof of \Cref{thm:t-vs}, is to show that all but $O(1)$ vertices of $G$ have weight $\Omega(1/t)$ (see \Cref{lem:almost-done}~\ref{itm:lowerbd}). The trick here is to compare the Lagrangians of $G$ and $G'$, where $G'$ is formed by removing from $G$ all edges with at least two vertices in the set $S$ of the $\Omega(1)$ vertices of smallest weight: on the one hand, we remove many edges and thus the Lagrangian drops significantly (using \Cref{lem:bigdiff}), but on the other hand, if all vertices in $S$ have weight $o(1/t)$ then $w(G')$ is not much smaller than $w(G) = \lambda(G)$, a contradiction.

		To finish the proof of \Cref{thm:t-vs}, we assume towards contradiction that $G$ has $n > t$ vertices, and consider $w(\overline{G})$ -- the weight of the complement graph $\overline{G} := [n]^{(r)} \setminus G$. We show that no vertex has very large degree in $\overline{G}$, which implies, by the key step, that $w(\overline{G})$ is large. However, as  $w(G) \le w([n]^{(r)}) - w(\overline{G}) \le \lambda([n]^{(r)}) - w(\overline{G})$, it follows that $w(G)$ is smaller than it should be, a contradiction.
	
	\subsection{The Frankl-F\"uredi conjecture holds in the principal range}
		In order to prove \Cref{thm:main1}, which states that the Frankl-F\"uredi conjecture holds in the principal range, it suffices to consider $m$ of the form $\binom{t}{r} - \binom{t-2}{r-2}$ with $t \in \N$, and to show that every $r$-graph $G$ with $m$ edges satisfies $\lambda(G) \le \lambda(\col{m}{r}) = \lambda([t-1]^{(r)})$ (see \Cref{rem:principal-range}). By \Cref{thm:t-vs}, we may assume that $V(G) = [t]$. Let $w$ be a weight function for $G$ with $w(G) = \lambda(G)$, and assume without loss of generality that $w$ is decreasing, i.e.\ $w(1) \ge \ldots \ge w(t)$. 
				
		We again start by proving simple facts about $G$ and $w$ (see \Cref{prop:prelims}). For example, we show that $w(x) = \Omega(1/t)$ for every $x \in [t-1]$ -- this fact and its proof are reminiscent of the key step in the previous proof. We use this, as well as known facts about Lagrangians, to obtain an estimate on the weight of any vertex $x \in [t-1]$ in terms of $w(t)$ and the number of absent edges incident with $x$.

		To conclude the proof, we consider the $r$-graph $H := \col{m}{r}$ -- this is the $r$-graph on vertex set $[t]$ whose non-edges are exactly the $r$-sets containing $\{t-1, t\}$. As $G$ and $H$ have the same number of edges, we can pair the elements of $E(G) \setminus E(H)$ with those of $E(H) \setminus E(G)$, and we think of $H$ as obtained from $G$ by swapping edges and non-edges that form pairs. We evaluate $w(G) - w(H)$ by evaluating the contribution of each swap, using the above estimate of the weight of a given vertex. We then use the symmetry of $H$ to show that by slightly modifying $w$ we are able to regain more weight than we lost, which implies that $w(H) > w(G)$ (unless $\lambda(G) \le \lambda([t-1]^{(r)})$), a contradiction to the choice of $G$. 

	\subsection{The Frankl-F\"uredi conjecture is false}

		In \Cref{sec:counter} we reveal a connection between the Lagrangian of an $r$-graph $G$ with certain properties, and the sum of degrees squared of a related $(r-2)$-graph $H$ (see \Cref{thm:counter-exs}). We exploit this connection to show that quite often outside of the principal range, the colex graph does not maximise the Lagrangian. 

		Consider an $r$-graph $G$ on vertex set $[t]$, whose edge set consists of all $r$-sets that do not contain $I := \{t-(i-1), \ldots, t\}$, where $2 \le i \le r-2$, together with some additional $r$-sets that we encode by an $(r-i)$-graph $H$ (namely, $H$ is an $(r-i)$-graph on vertex set $[t-i]$ whose edges are $(r-i)$-sets whose union with $I$ forms an edge of $G$). We observe that any colex graph $\col{m}{r}$ with $\binom{t}{r} - \binom{t-2}{r-2} < m < \binom{t}{r}$ and $t \in \N$ has this form. 

		Suppose that $G$ is such an $r$-graph that maximises the Lagrangian among $r$-graphs with $|G|$ edges. In \Cref{thm:counter-exs}, we prove that $H$ asymptotically maximises the sum of degrees squared among $(r-i)$-graphs with $t-i$ vertices and $|H|$ edges. Using this result, we show (see \Cref{subsec:counter-exs}) that the Frankl-F\"uredi conjecture is false for many values $m$, by demonstrating that the appropriate graph $H$ is far from maximising sum of degrees squared.

		To prove \Cref{thm:counter-exs}, we consider an $r$-graph $G'$ with the above structure, such that $|G'| = |G|$ and such that the $(r-i)$-graph $H'$, defined similarly to $H$, maximises the sum of degrees squared among $(r-i)$-graphs with $t-i$ vertices and $|H|$ edges. Given a weight function $w$ for $G$ with $w(G) = \lambda(G)$, we compare $w(G)$ with $w'(G')$, where $w'$ is a suitable weight function. Similarly to the previous proof, our first step is to estimate $w(x)$ in terms of the number of edges of $H$ incident with $x$ (the approach in the previous proof would allow us to estimate $w(x)$ in terms of the \emph{non-edges} of $H$). We use this estimate to define a suitable $w'$ by modifying $w$ appropriately. The remainder of the proof is a somewhat technical evaluation of $w'(G') - w(G)$ in terms of the difference between the sum of degrees squared of $H'$ and $H'$, where we consider separately the edges containing $I$ and the edges not containing $I$.

\section{Preliminaries}\label{sec:prelims}

	In this section we introduce definitions and notation that will be used throughout the paper, and prove some preliminary lemmas.     

	Say that $w = (w(x))_{x \in \N}$ is a \emph{weighting} of $\N$ if $w(x) \ge 0$ for all $x \in \N$ and $\sum_{x \in \N}w(x) = 1$. For $e \in \N^{(r)}$ and a weighting $w$ of $\N$, define 
	\[	
		w(e):= \prod_{x \in e}w(x),
	\]
	and for a finite $G \subseteq \N^{(r)}$ define
	\[
		w(G):= \sum_{e \in G}w(e).
	\]
	We may also say that $w$ is a \emph{weighting of $[t]$} if it is a weighting of $\N$ supported on $[t]$, or that $w$ is a \emph{weighting of $G$}, where $G \subseteq [t]^{(r)}$, if $w$ is a weighting of $[t]$.
	Say that a weighting $w$ of $\N$ is \emph{maximal} for $G$ if $w(G) = \lambda(G)$. 
	Also define
	\[
		\Lambda(m,r):= \max\{\lambda(G): G \subseteq \N^{(r)}, |G| = m\}.
	\]
	Using these definitions, $\lambda(G)$ can be expressed as $\max\{w(G): w\text{ is a weighting of $\N$}\}$ and \Cref{conj:TheConj} can be phrased as saying that there exists some weighting $w$ such that $w(\col{m}{r}) = \Lambda(m,r)$. 
	    	
	Let us introduce some more technical notation. 
	For $G \subseteq \N^{(r)}$, let $V(G)$ be the set of elements of $\N$ that are non-isolated vertices of $G$, i.e.\ that appear in some edge of $G$.
	For $S \subseteq V(G)$,
	define $N_G(S):=  \{e \setminus S: e \cup S \in G\}$; whenever $G$ is clear from the context we omit the subscript $G$. We may sometimes abuse notation and write $N(v_1,\ldots,v_s)$ when $S = \{v_1,\ldots,v_s\}$. For $x \in V(G)$, define $G\setminus \{x\}$ to be the hypergraph on vertex set $V(G) \setminus \{x\}$ and edge set $\{e \in G: x \notin e\}$. For vertices $x, y \in V(G)$, we define $N_y(x) := N(x) \setminus \{y\}$. 

	The first lemma we present gives some properties of any maximal weighting of $G$. 

    \begin{lem}[Frankl and R\"{o}dl~\cite{FraRod}]
        \label{lem:frank}
        Let $w$ be a maximal weighting of $G \subseteq \N^{(r)}$.  Then
        \begin{enumerate}
            \item\label{itm:nb-same} 
				For all $x,y \in V(G)$ with $w(x), w(y) > 0$, we have $w(N(x)) = w(N(y)).$
			\item\label{itm:opt-ij} 
				If $x,y \in V(G)$ are such that there is no edge of $G$ containing $\{x,y\}$, then $\lambda(G) \le \max\{\lambda(G \setminus\{x\}), \lambda(G \setminus \{y\})\}$. 
        \end{enumerate}
    \end{lem}
    
    \begin{proof}
		For \ref{itm:nb-same}, suppose, in order to obtain a contradiction, that $w(N(x)) > w(N(y))$. Let $0 < \eps < \min\{w(x),w(y)\}$. Define another weighting $w'$ of $G$ as follows.
		\[
			w'(z) = \left\{
				\begin{array}{ll}
					w(z) & z \neq x,y \\
					w(x) + \eps & z = x \\
					w(y) - \eps & z = y.
				\end{array}
				\right.
		\]
		As we have
		\begin{align*}
    		& w'(N(x)) = w(N(x)) - \eps \cdot w(N(x,y)) \,\,\,\text{ and } \\
    		& w'(N(y)) = w(N(y)) + \eps \cdot w(N(x,y)),
		\end{align*}
    	it follows that
    	\begin{equation}
    	\label{eqn:wvw}
    		w'(G) - w(G) 
			= \eps \cdot \left(w(N(x)) - w(N(y))\right) - \eps^2 \cdot w(N(x,y)).
    	\end{equation}
    	Choosing $\eps$ to be sufficiently small gives $w'(G) - w(G) >0$, contradicting the choice of $w$ as a maximal weighting of $G$. This completes the proof of \ref{itm:nb-same}.
    	
		Suppose, without loss of generality, that $w(x) \ge w(y)$ and that there is no edge containing $x$ and $y$ (and so $N(x,y) = \emptyset$). Define $w'$ as above with $\eps = w(y)$. From \eqref{eqn:wvw} and \ref{itm:nb-same} we see that $w'$ is a maximal weighting of $G$, where $w'(y) = 0$. But defining $G':= G \setminus \{y\}$, we get that $w'$ is a weighting of $G'$ such that $w'(G')  = w'(G)$. This proves \ref{itm:opt-ij}.	
    \end{proof}

	\begin{rem} \label{rem:principal-range}
		Recall that the principal range is the collection of intervals $\left[ \binom{t}{r}, \binom{t+1}{r} - \binom{t-1}{r-2} \right]$, with $t \in \N$. 
		This range is interesting because the $t$-th interval in this range is the set of values $m \in \N$ such that the Lagrangian of the colex graph $\col{m}{r}$ is the same as the Lagrangian of the clique $[t]^{(r)}$. Indeed, for $m$ in this interval, the colex graph $\col{m}{r}$ contains the clique $[t]^{(r)}$, so $\lambda(\col{m}{r}) \ge \lambda([t]^{(r)})$. Also, no edge of $\col{m}{r}$ contains both $t$ and $t+1$, thus by \Cref{lem:frank} $\lambda(\col{m}{r}) \le \lambda([t]^{(r)})$,  and so we have equality. It is easy to check that for $m = \binom{t+1}{r} - \binom{t-1}{r-2} + 1$ (the first value of $m$ outside the $t$-th interval in the principal range), $\lambda(\col{m}{r}) > \lambda([t]^{(r)})$, by considering the weighting $w$ of $[t+1]$ with $w(x) = \frac{1}{t}$ for $x \in [t-1]$ and $w(t+1) = w(t) = \frac{1}{2t}$. 
	\end{rem}

	\begin{cor} \label{cor:frankl}
		Let $w$ be a maximal weighting of $G \subseteq \N^{(r)}$, and let $x,y \in V(G)$ be such that $w(x), w(y) > 0$. Then 
		\begin{equation} \label{eq:frankl}
			w(N(x,y)) \cdot (w(x) - w(y)) = w(N_y(x)) - w(N_x(y)).
		\end{equation} 
	\end{cor}

	\begin{proof}
		Using the relation $w(N(x)) = w(y) \cdot w(N(x,y)) + w(N_y(x))$ and \Cref{lem:frank} \ref{itm:nb-same}, the proof follows.
	\end{proof}
    
    Given \Cref{lem:frank} \ref{itm:nb-same}, it is easy to see that for $G \subseteq \N^{(r)}$, any maximal weighting $w$ of $G$, and any $y \in V(G)$ with $w(y) > 0$, we can write 
	\begin{equation} \label{eqn:easylag}
		w(G) = \frac{1}{r}\sum_{x \in V(G)}w(x)w(N(x)) = \frac{w(N(y))}{r}.
	\end{equation}
    
    In order to state our next preliminary lemma, we should first introduce more definitions. Recall that for $x,y \in \N$ with $x < y$, the \emph{$xy$}-\emph{compression} of $F \in \N^{(r)}$ is defined to be
    \[
    	C_{xy}(F):= 
    		\begin{cases}
    			(F \setminus y) \cup x & \text{if } x \notin F, y \in F,\\
    			F & \text{otherwise}.
    		\end{cases}
	\]
    For $\F \subseteq \N^{(r)}$ we define
	\[
    	C_{xy}(\F):= \Big\{C_{xy}(F): F \in \F\Big\} \,\cup\, \Big\{F \in \F: C_{xy}(F) \in \F\Big\}.
	\] 
    $\F$ is said to be \emph{left-compressed} if $C_{xy}(\F) = \F$ for all $x < y$. 
    
    The next lemma tells us that to find the maximum value of $\lambda(G)$ over all $r$-graphs with $m$ edges, it suffices to consider left-compressed hypergraphs. We say that a weighting $w$ of $\N$ is \emph{decreasing} if $w(i) \ge w(j)$ for all $i < j$. 

    \begin{lem}[Frankl and F\"uredi~\cite{FraFur}]\label{lem:leftcomp}
        Let $G \subseteq \N^{(r)}$. For any decreasing weighting $w$ of $\N$ and any $x,y \in \N$ with $x < y$, we have $w(C_{xy}(G)) \ge w(G)$.
    \end{lem}
    \begin{proof}
    	We have
		\begin{equation} \label{eqn:compress}
    		w(C_{xy}(G)) - w(G) = \sum_{\substack{e \in G\\ C_{xy}(e) \notin G\\x \notin e,\, y \in e}} (w(x) - w(y)) \cdot w(e \setminus \{y\}).
		\end{equation}
    	As $w$ is decreasing, the right-hand side is non-negative. This completes the proof.
    \end{proof}

	By \Cref{lem:leftcomp}, in order to bound $\Lambda(m, r)$ it suffices to consider left-compressed $r$-graphs. In fact, the following lemma shows that every maximiser of the Lagrangian is left-compressed (after disregarding vertices of weight $0$ and relabelling).
	\begin{lem} \label{lem:leftcomp-strong}
		Let $G$ be an $r$-graph that maximises the Lagrangian among $r$-graphs with the same number of edges, and suppose that $w$ is a maximal weighting of $G$ which is decreasing and which satisfies $w(x) > 0$ for every $x \in V(G)$. Then $G$ is left-compressed.
	\end{lem}

	\begin{proof}
		Suppose that $G$ is not left-compressed. Then there exist vertices $x < y$ such that $G' := C_{xy}(G) \neq G$. By \Cref{lem:frank} \ref{itm:nb-same}, $w(N_G(x)) = w(N_G(y))$. Moreover, by the assumption on $x$ and $y$,  $N_{G'}(y) \subsetneqq N_G(y)$ and $N_{G'}(x) \supsetneqq N_G(x)$, thus $w(N_{G'}(x)) > w(N_{G'}(y))$. 
		But, by \Cref{lem:leftcomp}, $w(G') \ge w(G)$, and thus by maximality of $G$, $w$ is a maximal weighting of $G'$, so the inequality $w(N_{G'}(x)) > w(N_{G'}(y))$ contradicts \Cref{lem:frank} \ref{itm:nb-same}.
	\end{proof}

    We now collect some simple deductions about Lagrangians that are used throughout the paper.
	\begin{lem}
	\label{lem:lagclique}
		Let $r \ge 3$. The following statements hold.
		\begin{enumerate}
			\item \label{itm:lag-clique} 
				$\lambda([t]^{(r)}) = \frac{1}{t^r} \binom{t}{r} =  \frac{1}{r!}\left(1 - \frac{(r-1)r}{2t} +  O(t^{-2})\right)$ for $t \in \N$.
			\item \label{itm:clique-diff} 
				For $s = O(t)$, we have $\left|\lambda([t + s]^{(r)}) - \lambda([t]^{(r)})\right| = O\left(\frac{s}{t^2}\right)$.
			\item \label{itm:easy-ub} 
				$\lambda(G) \le 1/r!$ for every $r$-graph $G$.
		\end{enumerate}
	\end{lem}

	\begin{proof}
		Let $w$ be a maximal weighting of $[t]^{(r)}$. By \Cref{lem:frank} \ref{itm:nb-same}, we have $w(N(x)) = w(N(y))$ for all $x,y \in [t]$. By \Cref{cor:frankl}, as $N_x(y) = N_y(x)$ for every $x,y \in [t]$ in $[t]^{(r)}$, we have $w(x) = w(y)$ for all $x,y \in [t]$. Hence every vertex has weight $t^{-1}$.
		So 
		\[
			\lambda([t]^{(r)}) 
			= \frac{1}{t^r} \binom{t}{r} 
			= \frac{1}{r!}\left(1 - \frac{(r-1)r}{2t} +   O(t^{-2})\right),	
		\]
		as required for \ref{itm:lag-clique}.
		
		Now for \ref{itm:clique-diff}, using \ref{itm:lag-clique} we have 
		\[
			\left|\lambda([t + s]^{(r)}) - \lambda([t]^{(r)})\right| = \frac{1}{r!} \left( \frac{(r-1)r}{2t} - \frac{(r-1)r}{2(t + s)} + O(t^{-2})\right) = O\left(\frac{s}{t^2}\right),
		\]
		as required. 
		
		For \ref{itm:easy-ub}, let $t \in \N$ be such that $G \subseteq [t]^{(r)}$, and let $w$ be a maximal weighting of $G$. Then 
		\[
			\lambda(G) = w(G) \le w([t]^{(r)}) \le \lambda([t]^{(r)}),
		\]
		which is less than $1/r!$ by \ref{itm:lag-clique}. 	
	\end{proof}

\section{Bounding the support of $G$}\label{sec:supp}

	The aim of this section is to show that every $r$-graph that maximises the Lagrangian among $r$-graphs with $m \le \binom{t}{r}$ edges has at most $t$ vertices with non-zero weight.
	The main result of the section is the following theorem.
	
	\begin{thm}\label{thm:t-vs} 
		Let $r \ge 3$ and let $t \in \N$ be sufficiently large. Suppose that $G \subseteq \N^{(r)}$ is an $r$-graph with $\binom{t-1}{r} \le |G| \le \binom{t}{r}$ that maximises the Lagrangian among $r$-graphs of the same size, and suppose that $w$ is a maximal weighting of $G$ which is decreasing and satisfies $w(x) > 0$ for every non-isolated vertex of $G$. Then $G \subseteq [t]^{(r)}$.            
	\end{thm}

	In the next subsection we give a couple of easy corollaries of \Cref{thm:t-vs}. In \Cref{subsec:prelim-t-vs} we build up towards the proof of \Cref{thm:t-vs}, by giving a sequence of three preparatory results. The proof of \Cref{thm:t-vs} then follows quite easily, and is given in \Cref{subsec:proof-t-vs}.

	\subsection{Easy corollaries}

		Using \Cref{lem:leftcomp-strong}, we obtain the following corollary.
		\begin{cor} \label{cor:t-vs-compressed}
			Let $r \ge 3$, let $t \in \N$ be sufficiently large, and let $m$ satisfy $\binom{t}{r} - \binom{t-2}{r-2} < m \le \binom{t}{r}$. Suppose that $G$ maximises the Lagrangian among $r$-graphs with $m$ edges. Then $G$ is isomorphic to a left-compressed subgraph of $[t]^{(r)}$ with $m$ edges.
		\end{cor}

		\begin{proof}
			Let $w$ be a maximal weighting of $G$. Without loss of generality, $G \subseteq \N^{(r)}$ and $w$ is decreasing. By \Cref{thm:t-vs}, $G \subseteq [t]^{(r)}$. By maximality of $\lambda(G)$, we have $\lambda(G) \ge \lambda(\col{m}{r}) > \lambda([t-1]^{(r)})$ (for the second inequality, see \Cref{rem:principal-range}). It follows that the support of $w$ is exactly $[t]$ (i.e.\ $w(t) > 0$), because otherwise $w(G) \le \lambda([t-1]^{(r)})$, a contradiction. By the maximality assumption, $G$ consists of exactly $m$ edges in $[t]^{(r)}$, and by \Cref{lem:leftcomp-strong}, $G$ is left-compressed. 
		\end{proof}

		\Cref{cor:t-vs-compressed} immediately implies \Cref{conj:TheConj} in the principal case, and also when the number of edges is $1$ or $2$ below a principal case.

		\begin{obs}
			Let $r \ge 3$, let $t \in \N$ sufficiently large, let $a \in \{0,1,2\}$ and set $m:= \binom{t}{r} - a$. Suppose that $G$ is an $r$-graph that maximises the Lagrangian among $r$-graphs with $m$ edges. By \Cref{cor:t-vs-compressed}, $G$ is isomorphic to a left-compressed subgraph of $[t]^{(r)}$ on $m$ edges. Since there is only one such graph, namely $\col{m}{r}$, it follows that $G \simeq \col{m}{r}$.
		\end{obs}
	\subsection{Preparation for the proof of \Cref{thm:t-vs}} \label{subsec:prelim-t-vs}
		
		We begin our preparation for the proof of \Cref{thm:t-vs} by proving some simple facts about $r$-graphs that maximise the Lagrangian. We note that statements similar to \ref{itm:vx-bound} and \ref{itm:bigOt} were also proved in \cite{Tyo} (in Section 3). Recall that $\Lambda(m, r) = \max\{\lambda(G) : \text{$G$ is an $r$-graph with $m$ edges}\}$.	

		\begin{prop}\label{prop:bigO}
			Let $r \ge 3$, let $t \in \N$ be sufficiently large and let $m$ satisfy $\binom{t-1}{r} \le m \le \binom{t}{r}$.  
			Let $G \subseteq \N^{(r)}$ and let $w$ be a weighting of $G$ such that $|G| \le m$, $w(G) = \lambda(G) = \Lambda(m, r)$, and $w(x) > 0$ for every $x \in V(G)$. Then the following properties hold, where $\rho, \kappa > 0$ are constants that depend only on $r$.
			\begin{enumerate}
				\item\label{itm:vx-bound} 
					$w(x) = O(t^{-1})$ for every $x \in V(G)$.
				\item\label{itm:nb-bound} 
					$|N(x)| \ge  \rho t^{r-1}$ for all $x \in V(G)$.
				\item\label{itm:bigOt} $|V(G)| \le \kappa t$.
				\item\label{itm:many-vx-big} 
					$\Omega(t)$ vertices of $G$ have weight $\Omega(t^{-1})$.  
			\end{enumerate} 		
		\end{prop}

		\begin{proof}
		
			By \Cref{lem:lagclique} \ref{itm:lag-clique}, and because $G$ maximises the Lagrangian among $r$-graphs of size at most $m > \binom{t-1}{r}$, 	 
			\begin{equation} \label{eqn:lag-G-lower}	
				\lambda(G) \ge \lambda([t-1]^{(r)}) = \frac{1}{r!} \cdot (1 + O(t^{-1})).
			\end{equation}
			Note that $w(N(x)) \le (1 - w(x))^{r-1} \cdot \lambda(N(x))$, which is at most $(1 - w(x))^{r-1} \cdot \frac{1}{(r-1)!}$ by \Cref{lem:lagclique} \ref{itm:easy-ub}.
			Using \eqref{eqn:easylag}, we thus have
			\[
				\lambda(G) = \frac{w(N(x))}{r} \le \frac{1}{r}\cdot (1 - w(x))^{r-1} \cdot \frac{1}{(r-1)!} = \frac{1}{r!} \cdot (1 - w(x))^{r-1}.
			\]
			Putting the two inequalities together, we find that $w(x) = O(t^{-1})$, completing the proof of \ref{itm:vx-bound}.

			Now for \ref{itm:nb-bound}.  Using \ref{itm:vx-bound} and \eqref{eqn:lag-G-lower} gives that for each $x \in V(G)$:
			\begin{equation}
				\frac{1}{r!} \cdot (1 + O(t^{-1})) 
				\le \lambda(G) 
				= \frac{w(N(x))}{r} 
				= O(|N(x)|t^{-(r-1)}).
			\end{equation}
			By rearranging, we find that $|N(x)| = \Omega( t^{r-1})$, as required for \ref{itm:nb-bound}.	
			
			We have	
			\[ 
				\frac{1}{r}\sum_{x \in V(G)}|N(x)| \le m \le \binom{t}{r}.
			\]
			Using \ref{itm:nb-bound} to bound $|N(x)|$ and rearranging gives $|V(G)| = O(t)$, as required for \ref{itm:bigOt}.
							
			Let $\kappa$ be as in \ref{itm:bigOt}; we assume, without loss of generality, that $\kappa \ge 1$.
			Set $\omega := \max \{w(x) : x \in V(G)\}$, so by \ref{itm:vx-bound}, $\omega = O(t^{-1})$.
			Let $\delta := \min\{\kappa,(2\kappa t \cdot \omega)^{-1}\}$ and note that $\delta = \Omega(1)$. Suppose, for a contradiction, that at most $\delta t$ vertices have weight at least $\frac{1}{2\kappa^2 t}$. Then by the choice of $\kappa$, $\omega$ and $\delta$, 
			\[
				\sum_{x \in V(G)}w(x) 
				\le \delta t \cdot \omega + (\kappa - \delta)t \cdot \frac{1}{2\kappa^2 t} 
				< \frac{1}{2\kappa} + \frac{1}{2\kappa} 
				< 1,
			\]
			a contradiction. This proves \ref{itm:many-vx-big}. 
		\end{proof}

		Throughout the remainder of the section, define $\rho$ and $\kappa$ to be the constants from \Cref{prop:bigO}. 
		The next lemma shows that removing $\Omega(t^{r-1})$ edges from an $r$-graph $H$ with $\lambda(H) = \Lambda(|H|,r)$ will decrease the Lagrangian by $\Omega(t^{-2})$.
				
		\begin{lem}\label{lem:bigdiff}
			Let $r \ge 3$, let $t \in \N$ be sufficiently large, and let $m$ satisfy $\binom{t-1}{r} \le m \le \binom{t}{r}$. There exists $\alpha > 0$ such that
				$$\Lambda(m,r) - \Lambda(m- \alpha t^{r-1}, r) = \Omega(t^{-2}).$$
		\end{lem}		
					
		\begin{proof}
				
			Define $\alpha:= 2\kappa^{r-1}$ and let $t$ be large enough so that 
			\begin{equation}
				\label{eqn:pickm}
				m' := m - \alpha t^{r-1} 
				> \binom{\frac{t}{2}}{r}.
			\end{equation}
			Let $H \subseteq \N^{(r)}$ and let $w$ be a weighting of $H$ such that $|H| \le m'$, $w(H) = \lambda(H) = \Lambda(m', r)$, and $w(x) > 0$ for every $x \in V(H)$.
			
			By \Cref{prop:bigO} \ref{itm:bigOt}, the definition of $m'$, and \eqref{eqn:pickm}, we have $|V(H)| \le \kappa t$. 
			Let $x \in V(H)$ be such that $w(x) \ge w(y)$ for all $y \in V(H)$. Let $u$ be a new vertex and define a new $r$-graph $H'$ by setting 
			\[	
				H' := H \,\cup\, \Big\{\{u\} \cup f: f \in N(x)\Big\} \,\cup\, \Big\{\{x,u\} \cup S: S \subseteq V(H)\setminus \{x\}, |S| = r-2 \Big\}.
			\]
			Define a weighting $w'$ of $H'$ as follows. 
			\[
				w'(v) = \left\{
					\begin{array}{ll}
						w(v) & v \neq x, u \\
						\frac{1}{2} \cdot w(x) & v \in \{x, u\}.
					\end{array}
					\right.
			\]
			For $t$ sufficiently large, we have (crudely, using $|V(H)| \le \kappa t$)
			\begin{equation} \label{eqn:alpha}
				|H'| - |H| \le \kappa^{r-1} t^{r-1} + \kappa^{r-2}t^{r-2} < \alpha t^{r-1}.
			\end{equation}
			We also have that $w'(H') - w(H)$ is precisely the weight (with respect to $w'$) of the edges in $H'$ containing $x$ and $u$. By \Cref{prop:bigO} \ref{itm:many-vx-big} and \eqref{eqn:pickm}, $\Omega(t)$ vertices of $H$ have weight $\Omega(1/t)$. Let $B$ be the set of these vertices and let $E_B \subseteq H'$ be the edges that contain $\{x, u\}$ and are contained within $B \cup \{x,u\}$. So we have $w'(E_B) = \Omega(t^{-2})$. 
		
			By \eqref{eqn:pickm} and \eqref{eqn:alpha} we have $|H'| \le m$. So putting this all together gives that 
			\[	
				\Lambda(m ,r) - \Lambda(m- \alpha t^{r-1}, r) \ge \lambda(H') - \lambda(H) = \Omega(t^{-2}),
			\]
			as required.
		\end{proof}	
		
		Given an $r$-graph $G$ that maximises the Lagrangian, we now give a lower bound on the weight of all but $O(1)$ of its vertices. 
			
		\begin{lem}\label{lem:almost-done}
			Let $r \ge 3$, $t \in \N$ and $m$ be such that $\binom{t-1}{r} \le m \le \binom{t}{r}$. Suppose that $G \subseteq [n]^{(r)}$ maximises the Lagrangian among $r$-graphs with $m$ edges. Let $w$ be a maximal weighting of $G$ and suppose that $w(x) > 0$ for every $x \in [n]$.
			The following statements hold.
			\begin{enumerate}
				\item\label{itm:lowerbd} 
					There exists a constant $\beta > 0$ such that all but $O(1)$ vertices $x$ in $G$ satisfy $w(x) \ge \frac{\beta}{t}$. 
				\item\label{itm:nbdiff} 
					For all $x,y \in [n]$, we have $\big||N(x)| - |N(y)|\big| = O(t^{r-2})$.
			\end{enumerate}
		\end{lem}

		\begin{proof}
			We assume that $w$ is decreasing. It follows from \Cref{lem:leftcomp-strong} that $G$ is left-compressed.

			Let $\alpha$ be the constant from \Cref{lem:bigdiff} and let $\gamma := \ceil{\frac{2\alpha}{\rho}}$. We will show that $w(n-\gamma) = \Omega(t^{-1})$, from which \ref{itm:lowerbd} will follow as $w$ is decreasing.
			Let $S:= \{n-\gamma,\ldots,n\} \subseteq V(G)$ and let $G'$ be the graph obtained from $G$ by deleting edges incident with at least two vertices of $S$. Let $W$ be the weight of the edges containing at least two vertices of $S$. As $w$ is decreasing and $w(N(x,y)) \le 1$, 
			\begin{equation}\label{upper}
				W \le \sum_{x,y \in S} w(x)w(y)w(N(x,y)) \le \binom{\gamma + 1}{2} \cdot w(n-\gamma)^2.
			\end{equation}
			As there are no edges in $G'$ that contain a pair of vertices from $S$, by \Cref{lem:frank} \ref{itm:opt-ij}, there exists some $U \subseteq S$ with $|U| = |S|-1$ such that $\lambda(G') = \lambda(G'')$, where $G'':= G \setminus U$. By choice of $\gamma$, 

			\[
				|G| - |G''| 
				= \sum_{u \in U} |N(u)| + O(t^{r-2}) 
				\ge 2 \alpha t^{r-1} + O(t^{r-2}) 
				\ge \alpha t^{r-1}.
			\]
			So using \Cref{lem:bigdiff}, we find that $\lambda(G) - \lambda(G'') = \Omega(t^{-2})$.
			Combining this with \eqref{upper} gives 
			\[
				\binom{\gamma+1}{2} \cdot w(n-\gamma)^2 \ge W = w(G) - w(G') \ge \lambda(G) - \lambda(G'') = \Omega(t^{-2}) .
			\]
			This shows that there exists a constant $\beta >0$ such that $w(n-\gamma) \ge \frac{\beta}{t}$, as required for \ref{itm:lowerbd}.
			
			Now consider \ref{itm:nbdiff}. As $G$ is left-compressed, we may assume that $N_x(y) \subseteq N_y(x)$. Thus we have
			\[
				0 \le |N(x)| - |N(y)| = |N_y(x)| - |N_x(y)| = |N_y(x) \setminus N_x(y)|.
			\]
			Let $T:= N_y(x) \setminus N_x(y)$, let $T_1 := \{e \in T: w(y) \ge \frac{\beta}{t} \text{ for all } y \in e\}$ and let $T_2 := T \setminus T_1.$ We will show that $|T_1|, |T_2| = O(t^{r-2})$, which will imply the claim. 
			 
			First consider $|T_2|$. From \ref{itm:lowerbd} we know that at most $\gamma$ vertices have weight less than $\frac{\beta}{t}$. By \Cref{prop:bigO} \ref{itm:bigOt}, $G$ has at most $\kappa t$ vertices. So $|T_2| \le \gamma \cdot (\kappa t)^{r-2} = O(t^{r-2})$, as required. 	                     
			 
			Now let us bound $|T_1|$. Using \Cref{prop:bigO} \ref{itm:vx-bound} shows that we can bound the left-hand side of \eqref{eq:frankl} by
			\begin{equation}
			\label{eqn:LHSuseful}
				w(N(x,y))(w(x) - w(y)) 
				\le  |N(x,y)| \cdot O(t^{-(r-1)})
				\le \binom{\kappa t}{r-2}\cdot O(t^{-(r-1)})
				= O(t^{-1}).
			\end{equation}
			By definition of $T_1$, we have $w(T_1) \ge \left(\frac{\beta}{t}\right)^{r-1}|T_1|$, so we can bound the right-hand side of \eqref{eq:frankl} by
			\begin{equation}
			\label{eqn:RHSuseful}
				w(N_y(x)) - w(N_x(y)) = w(T) \ge w(T_1) \ge \left(\frac{\beta}{t}\right)^{r-1}|T_1|.
			\end{equation}
			Combining \eqref{eqn:LHSuseful} and \eqref{eqn:RHSuseful} gives that $|T_1| = O(t^{r-2})$, as required. This completes the proof of \ref{itm:nbdiff} and the proof of the lemma.
		\end{proof}	
	\subsection{Proof of \Cref{thm:t-vs}} \label{subsec:proof-t-vs}
			
		We now combine the results from the previous subsection to prove \Cref{thm:t-vs}. 

		\begin{proof}[Proof of \Cref{thm:t-vs}]

			As usual, we assume that $w$ is decreasing. Then $G \subseteq [n]^{(r)}$ for some $n \in \N$, and we may assume that $w(x) > 0$ for every $x \in [n]$. Also, by \Cref{lem:leftcomp-strong}, $G$ is left-compressed.
			Let us suppose, in order to obtain a contradiction, that $n = t+s$, for $s \ge 1$. We will show that the following holds, where $\overline{G} := [n]^{(r)} \setminus G$ and $\beta$ is the constant from \Cref{lem:almost-done}~\ref{itm:lowerbd}.
			\begin{equation}
				\label{eqn:lots-missing-are-big}
				\left|\{e \in \overline{G}: w(e) \ge (\beta / t)^{r}\}\right| = \Omega(s \cdot t^{r-1}).
			\end{equation}
			
			Before proving \eqref{eqn:lots-missing-are-big}, let us show how this implies the theorem. 
			By \eqref{eqn:lots-missing-are-big}, we have
			\[
				\lambda([n]^{(r)}) - \lambda(G) 
				\ge w(\overline{G}) 
				\ge \beta^r \cdot t^{-r} \cdot \Omega( s \cdot t^{r-1})
				= \Omega(s/t).
			\]
			However, by choice of $G$ we have $\lambda(G) \ge \lambda([t-1]^{(r)})$, so this contradicts \Cref{lem:lagclique} \ref{itm:clique-diff}. It follows that $n \le t$, as required. 
				
			It remains to prove \eqref{eqn:lots-missing-are-big}.	
			As $n = t+ s$, we have
			\[	
				\binom{t+s}{r} - \binom{t}{r}  \le	|\overline{G}| \le \binom{t+s}{r} - \binom{t-1}{r}.
			\]
			And so, as $|V(G)| = O(t)$ (by \Cref{prop:bigO} \ref{itm:bigOt}) and $s = O(t)$, 
			\[
				|\overline{G}| = \Theta(s \cdot t^{r-1}).
			\]
			Define				
			\[	
				U_1 := \{e \in \overline{G}: w(e) \ge \left(\beta/t\right)^{r}\} \hspace{0.5cm} \text{ and } \hspace{0.5cm}
				U_2 := \overline{G} \setminus U.
			\]
			Now suppose, in order to obtain a contradiction, that $|U_2| \ge |\overline{G}|/2$. 

			Let $S:= \{x \in V(G): w(x) < \beta/t\}$. By \Cref{lem:almost-done} \ref{itm:lowerbd}, $|S| = O(1)$. Each set in $U_2$ contains a vertex of $S$ and so, by the pigeonhole principle, some vertex $x \in S$ is contained in at least $|U_2| / |S| = \Omega(s \cdot t^{r-1})$ members of $U_2$. In particular, there are $\Omega(s \cdot t^{r-1})$  sets of $\overline{G}$ that contain $x$.

			However, using \Cref{lem:almost-done} \ref{itm:nbdiff}, gives that for all $y \in V(G)$ there are $\Omega(s \cdot t^{r-1})$ sets of $\overline{G}$ containing $y$. So in total, as $|V(G)| \ge t-1$, we have
			\[	
				|\overline{G}| \ge |V(G)| \cdot \Omega(s \cdot t^{r-1}) = \Omega( s \cdot t^{r})  > |\overline{G}|,
			\]
			a contradiction. It follows that $|U_2| \le |\overline{G}|/2$, and thus $|U_1| \ge |\overline{G}|/2 = \Omega(s \cdot t^{r-1})$. This completes the proof of \eqref{eqn:lots-missing-are-big} and hence the proof of \Cref{thm:t-vs}.
		\end{proof}
\section{Proof of \Cref{thm:main1}}\label{sec:a-big}
	The main goal of this section is to prove the following theorem.

	\begin{thm} \label{thm:clique}
		Let $r \ge 3$, and let $t \in \N$ be sufficiently large. Let $G$ be an $r$-graph with $\binom{t}{r} - \binom{t-2}{r-2}$ edges. Then $\lambda(G) \le \lambda([t-1]^{(r)})$.
	\end{thm}

	This implies \Cref{thm:main1} (as adding edges to a graph cannot decrease its Lagrangian). We begin by proving some simple bounds. The proof of \Cref{thm:clique} will then be given in Subsection~\ref{subsec:clique}. Our hope is that, by initially separating out these basic bounds, the key ideas within the proof will be clearer to the reader.  

	\subsection{Preliminaries}

		For a graph $G \subseteq [t]^{(r)}$, write $\overline{G} := [t]^{(r)} \setminus G$ (note that this definition depends on $t$, but since $t$ will always be clear from the context, we hope this imprecise notation will be clear). For a vertex $x \in [t]$, define $e(x)$ to be the number of edges of $\overline{G} := [t]^{(r)} \setminus G$ that contain $x$. 
					
		\begin{prop} \label{prop:prelims} 
			Let $r \ge 3$, let $t \in \N$ be sufficiently large, let $a$ be such that $0 \le a \le \binom{t-2}{r-2}$  and set $m := \binom{t}{r} - a$.
			Suppose that $G$ is a subgraph of $[t]^{(r)}$ with at most $m$ edges, and suppose that $G$ is left-compressed and that it maximises the Lagrangian among $r$-graphs with $m$ edges. Let $w$ be a maximal weighting of $G$.
			The following statements hold.
			\begin{enumerate}
				\item \label{itm:useful-missing-1}
					$\me{1} \le \frac{ra}{t} = O(t^{r-3})$.
				\item \label{itm:useful-w-t-1}
					$w(t-1) \ge \frac{1}{6t} + O(t^{-2})$.
				\item \label{itm:useful-w-x}
					$w(t) = w(1) - \Theta\left(\frac{\me{t} - \me{1}}{t^{r-1}}\right)$ and, for $x < t$, $w(x) = w(1) - \Theta\left(\frac{\me{x} - \me{1}}{t^{r-2}}\cdot w(t)\right)$.
			\end{enumerate}
		\end{prop}

		\begin{proof}
			We note first that if $w(x) = 0$ for some $x \in [t]$, then $\lambda(G) \le \lambda([t-1]^{(r)})$, which, by maximality of $G$, implies that $a = \binom{t-2}{r-2}$ (otherwise, by \Cref{rem:principal-range}, $G$ does not maximise the Lagrangian among $r$-graphs with $m$ edges). In this case $G$ must in fact be the graph $[t-1]^{(r)}$, by maximality of $G$ and the assumption that $G$ is left-compressed, and $w(x) = 1/(t-1)$ for every $x \in [t-1]$. It is easy to check that all the required statements hold in this case. Thus, we may assume that $w(x) > 0$ for every $x \in [t]$. By the assumption that $G$ is left-compressed, $w$ is decreasing.

			Statement \ref{itm:useful-missing-1} follows from the fact that $\me{1} \le \me{x}$ for every $x \in [t]$, which is a consequence of the assumption that $G$ is left-compressed.
			
			Now consider \ref{itm:useful-w-t-1}. Let $G'$ be obtained from $G$ by removing any edge containing at least two of the three vertices $t, t-1, t-2$. By \Cref{lem:frank} \ref{itm:opt-ij}, $\lambda(G') \le \lambda([t-2]^{(r-2)})$. Since $G$ maximises the Lagrangian, we have $w(G) \ge \lambda([t-1]^{(r)})$. Hence, by \Cref{lem:lagclique} \ref{itm:lag-clique}, $w(G) - w(G')$ is at least $\frac{1}{2(r-2)!} \cdot \frac{1}{t^2} + O(t^{-3})$, but is also at most 
			\[
				\frac{1}{(r-2)!} \cdot \Big(w(t)w(t-1) + w(t)w(t-2) + w(t-1)w(t-2)\Big) \le \frac{3}{(r-2)!} \cdot w(t-1)w(t-2).
			\]
			Since $w(t-2) \le \frac{1}{t-2}$, we find that $w(t-1) \ge \frac{1}{6t} + O(t^{-2})$, as required for \ref{itm:useful-w-t-1}.
		
			We now prove \ref{itm:useful-w-x}. First, we claim that $w(N(1, x)) = \Theta(1)$. Indeed, by \ref{itm:useful-missing-1}, there are at most $O(t^{r-3})$ non-edges in $N(1,x)$ (as a graph on vertex set $[t] \setminus \{1,x\}$), and by \ref{itm:useful-w-t-1}, the weight of each edge in $N(1,x)$, except for possibly those containing $t$, is $\Omega(t^{-(r-2)})$. Hence, 
			\[	
				w(N(1,x)) = \Omega\left( \left(\binom{t-3}{r-2} -  O(t^{r-3})\right) \cdot t^{-(r-2)} \right) = \Omega(1).
			\]
			For an upper bound, it follows from \Cref{lem:lagclique} \ref{itm:lag-clique} that
			\[	
				w(N(1,x)) \le \lambda([t-2]^{(r-2)}) \le 1/(r-2)! = O(1).
			\]
			
			\begin{claim}\label{cl:wt-missing}
				The weight of each absent edge of $G$ is $\Theta(w(t) \cdot t^{-(r-1)})$.
			\end{claim}
			\begin{proof}
				First consider the case where $w(t) > \frac{1}{2\cdot (6t)^r \cdot w(1)^{r-1}}$ (so $w(t) = \Omega(t^{-1})$ by \Cref{prop:bigO} \ref{itm:vx-bound}). By \Cref{prop:bigO} \ref{itm:vx-bound} and by \ref{itm:useful-w-t-1}, all vertices other than $t$ have weight $\Theta(t^{-1})$. It follows that, in this case, the weight of each absent edge is $\Theta(t^{-r}) = \Theta(w(t) t^{-(r-1)})$. 
				
				We now show that if $w(t) \le \frac{1}{2\cdot (6t)^r \cdot w(1)^{r-1}}$ then all absent edges contain $t$. Indeed, by maximality of $w(G)$, the weight of any absent edge is at most the weight of any existing edge. Now, if there is an absent edge that does not contain $t$ then it has weight at least $\frac{1}{(6t)^r} + O(t^{-(r+1)})$, by \ref{itm:useful-w-t-1}, which is larger than the weight of any $r$-set that contains $t$ (as the weight of any such $r$-set is at most $w(t) \cdot w(1)^{r-1} \le \frac{1}{2 \cdot (6t)^r}$), so all $r$-sets that contain $t$ are absent edges of $G$. But this implies that the number of absent edges is larger than $\binom{t-1}{r-1}$, a contradiction. So in this case, all absent edges contain $t$, as claimed. It follows that the weight of each absent edge is $\Theta(w(t) t^{-(r-1)})$. 
			\end{proof}
			As $G$ is left-compressed, $N_1(x) \subseteq N_x(1)$, so $|N_x(1)\setminus N_1(x)| = \me{x} - \me{1}$. By \Cref{cl:wt-missing}, each absent edge has weight $\Theta(w(t) t^{-(r-2)})$. Thus we have $w(N_x(1)) - w(N_1(x)) = \Theta\left(\frac{\me{x} - \me{1}}{t^{r-1}}\cdot w(t)\right)$, and  as $w(N(1, x)) = \Theta(1)$, using \Cref{cor:frankl},
			\[
				w(x) = w(1) - \Theta\left(\frac{\me{x} - \me{1}}{t^{r-1}}\cdot w(t)\right),
			\]
			as required for the second part of \ref{itm:useful-w-x}. The first part follows similarly, using $w(t-1) = \Omega(t^{-1})$ (see \ref{itm:useful-w-t-1}).
		\end{proof}

	\subsection{Proof of \Cref{thm:clique}}\label{subsec:clique}
		
		We are now ready to complete the proof of \Cref{thm:clique}. Before giving the details, we provide a brief overview. Define $H$ to be the $r$-graph on vertex set $[t]$ whose non-edges are exactly the $r$-tuples that contain $t-1$ and $t$ (so $H = \col{m}{r}$, where $m = \binom{t}{r} - \binom{t-2}{r-2}$).
		As $G$ and $H$ have the same number of edges, we can pair the elements of $E(G) \setminus E(H)$ with the elements of $E(H) \setminus E(G)$, and think of $H$ as obtained from $G$ by swapping edges and non-edges that form pairs. We evaluate $w(G) - w(H)$ by thinking of $H$ in this way and evaluating the contribution of each swap. Then we use the symmetry of $H$ to show that by slightly modifying $w$, we are able to regain more weight than we lost, thus showing that $\lambda(H) > \lambda(G)$, a contradiction to the choice of $G$.
		
		\begin{proof}[Proof of \Cref{thm:clique}]

			Let $G$ be an $r$-graph with at most $m := \binom{t}{r} - \binom{t-2}{r-2}$ edges. Suppose that $G$ maximises the Lagrangian among $r$-graphs with $m$ edges, and let $w$ be a maximal weighting of $G$. Without loss of generality, we assume that $w(x) > 0$ for every $x \in V(G)$. It follows from \Cref{thm:t-vs} that, without loss of generality, $G \subseteq [t]^{(r)}$.			
			If there is a vertex $x \in [t]$ with $w(x) = 0$, then $\lambda(G) \le \lambda([t-1]^{(r)})$, completing the proof of the theorem. Thus, we may assume that $w(x) > 0$ for every $x \in [t]$. It follows that $G$ has exactly $m$ edges, because adding any non-edge in $[t]^{(r)}$ would increase the weight of $G$. We may also assume that $w$ is decreasing, and that $G$ is left-compressed, by \Cref{lem:leftcomp-strong}.

			We may further assume that some non-edge of $G$ does not contain both $t$ and $t-1$; otherwise, the non-edges of $G$ are exactly those that contain both $t$ and $t-1$, so by \Cref{lem:frank} \ref{itm:opt-ij} we may remove one of $t$ and $t-1$ without decreasing $\lambda(G)$, but then we obtain a graph on $t-1$ vertices, hence $\lambda(G) \le \lambda([t-1]^{(r)})$, as required.

			Let $H = \col{m}{r}$, i.e.\ $H$ is the subgraph of $[t]^{(r)}$ obtained by removing all edges that contain both $t-1$ and $t$.
			
			\begin{claim} \label{cl:weight-loss-swaps}
				$w(G) - w(H) = O(t^{-0.1}w(t)^2)$.
			\end{claim}
				
			\begin{proof}
				We pair $r$-sets of $G \setminus H$ with $r$-sets of $H \setminus G$ (note that $|G| = |H|$, so such a pairing exists). Each such pair consists of an edge of $G$ that does not contain $\{t-1, t\}$ and a non-edge of $G$ that does contain $\{t-1, t\}$. In total there are at most $\binom{t-2}{r-2}$ such pairs.
				We think of $H$ as obtained from $G$ by a series of swaps of a non-edge of $G$ with an edge of $G$ (given by these pairs), and thus to evaluate $w(G) - w(H)$ we will estimate the weight lost by each of these swaps.

				For $x_1 < \ldots < x_r$ and $y_1 < \ldots < y_{r-2}$, let $x:= (x_1, \ldots, x_r)$ be a non-edge of $G$ not containing $\{t-1,t\}$ and let $y:=(y_1, \ldots, y_{r-2}, t-1, t)$ be an edge of $G$. The weight lost by swapping $x$ for $y$ is bounded from above by the following expression.
				\begin{align} \label{eqn:loss-swap}
					\begin{split}
						& w(y_1, \ldots, y_{r-2}, t-1, t) - w(x_1, \ldots, x_r) \\
						\le \, & \big(w(1)^{r-2} - w(x_{r-2})^{r-2}\big) \cdot w(t-1)w(t) \\
						\le\, &  \left(w(1)^{r-2} - \left(w(1) - O\left(\frac{\me{x_{r-2}}}{t^{r-2}}\cdot w(t)\right)\right)^{r-2}\right) \cdot w(t) w(t-1)  \\
						=\, & O\big(\me{x_{r-2}}  t^{-2(r-2)} w(t)^2\big).
					\end{split}
				\end{align}
				Here we used \Cref{prop:bigO} \ref{itm:vx-bound} and \Cref{prop:prelims} \ref{itm:useful-w-x}.
					
				If $\me{t-1} \le t^{r-2-0.1}$, then also $\me{x_{r-2}} \le t^{r-2-0.1}$ (as $x_{r-2} < t-1$), so the loss from one swap is $O(t^{-(r-2) - 0.1} w(t)^2)$. Thus in total $w(G) - w(H) = O(t^{-0.1}w(t)^2)$, as required.
					
				Now suppose that $\me{t-1} \ge t^{r-2-0.1}$. Let $S := \{x \in [t] : \me{x} \ge t^{-0.1} \cdot e(t-1)\}$.  We have $|S| = O(t^{0.2})$, as the total number of absent edges is $O(t^{r-2})$. We claim that every non-edge of $G$ contains at least two vertices from $S$. Indeed, let $(x_1, \ldots, x_r)$ be an absent edge, and suppose that it contains at most one vertex from $S$. Then, by \Cref{prop:prelims}~\ref{itm:useful-w-x},
				\begin{align} \label{eqn:e-t-2-large}
				\begin{split}
					w(x_1, \ldots, x_r) 
					& \ge \left(w(1) - O\left(\frac{ t^{-0.1} \cdot \me{t-1}}{t^{r-2}} \cdot w(t)\right)\right)^{r-1} \cdot w(t) \\
					& = \left(w(1) - O\left(\frac{ t^{-0.1} \cdot \me{t-1}}{t^{r-2}} \cdot w(t)\right)\right) \cdot w(1)^{r-2} \cdot w(t) \\
					& > w(1)^{r-2} \cdot w(t-1) \cdot w(t).
				\end{split}
				\end{align}
				It follows that the weight of $G$ can be increased by swapping $(x_1, \ldots, x_r)$ with any existing edge of $G$ that contains $t$ and $t-1$, a contradiction.
					
				We split the absent edges into two types: non-edges with exactly two vertices in $S$ (type 1), and non-edges with at least three vertices in $S$ (type 2).  
					
				By \eqref{eqn:loss-swap} and as $x_{r-2} \notin S$, the loss per swap of a type 1 non-edge is $O(t^{-0.1} \cdot e(t-1) \cdot t^{-2(r-2)} w(t)^2) = O(t^{-(r-2)-0.1} w(t)^2)$. Hence the total loss from swaps of non-edges of the first type is $O(t^{-0.1} w(t)^2)$. The loss from a swap of a type 2 non-edge is $O(t^{-(r-2)} w(t)^2)$. Note that, as the number of non-edges containing any fixed three vertices from $S$ is at most $\binom{t-3}{r-3}$, the number of non-edges of the second type is $O(|S|^3 \cdot t^{r-3}) = O(t^{r-3 + 0.6})$. Hence the loss from such edges is $O(t^{-0.4} w(t)^2)$. So the total loss is $O(t^{-0.1} w(t)^2)$, as required.
			\end{proof}

			We define a new weighting $w'$ of $[t]$ by 
			\[	
				w'(x) = \left\{ 
					\begin{array}{ll}
						w(x) & x \neq t,t-1 \\
						w(t-1) + w(t) & x = t-1  \\
						0 & x = t.
					\end{array}
				\right.
			\]
			As $H$ has no edges containing both $t$ and $t-1$ and the neighbourhoods of $t$ and $t-1$ in $[t-2]$ are the same, $w'(H) = w(H)$.
				
			\begin{claim} \label{cl:big-diff}
				There is a vertex $x \in [t-2]$ such that $|w(x) - w'(t-1)| = \Omega(t^{-0.02}w(t))$. 
			\end{claim}
				
			\begin{proof}
				If $\me{t-1} \le t^{r-2 - 0.01}$, then using \Cref{prop:prelims} \ref{itm:useful-w-x} we have $w(t-1) = w(1) - O(t^{-0.01}w(t))$. Hence $w'(t-1) - w(1) = w(t) + w(t-1) - w(1) = \Omega(w(t))$. Otherwise, if $\me{t-1} \ge t^{r-2 - 0.01}$, then $\me{t-2} \ge t^{r-2 - 0.02}$ (otherwise, similarly to the calculation in \eqref{eqn:e-t-2-large}, all non-edges contain $t-1$ and $t$, a contradiction to our assumptions on $G$). Hence $w(t-2) = w(1) - \Omega(t^{-0.02}w(t))$. So either $|w(1) - w'(t-1)| = \Omega(t^{-0.02}w(t))$, or $|w(t-2) - w'(t-1)| = \Omega(t^{-0.02}w(t))$. 
			\end{proof}
				
			Let $x_0$ be as in \Cref{cl:big-diff}. We define a new weighting $w''$ of $[t]$ by 
			\[	
				w''(x) = \left\{ 
					\begin{array}{ll}
						w'(x) & x \neq x_0,t-1 \\
						\frac{1}{2} \cdot \big(w'(t-1) + w'(x_0)\big) & x = t-1 \text{ or } x_0.
					\end{array}
				\right.
			\]
				
			Note that the total weight of the $r$-sets in $[t-1]$ with at most one vertex in $\{x_0, t-1\}$ is the same under $w'$ and $w''$. Hence
			\begin{align*}
				w''(H) - w'(H) 
				& = w(N_H(x_0,t-1)) \cdot \big(w''(x_0)w''(t-1) - w'(x_0) w'(t-1) \big) \\
				& = w(N_H(x_0,t-1)) \cdot \Big(\frac{1}{4}\cdot\big(w'(x_0) + w'(t-1)\big)^2 - w'(x_0) w'(t-1) \Big) \\
				& = w(N_H(x_0,t-1)) \cdot \frac{1}{4} \cdot \big(w(x_0) - w'(t-1)\big)^2 \\
				& = \Omega\left(t^{-0.04}w(t)^2\right), 
			\end{align*}
			where to bound $N_H(x_0,t-1)$ we used the fact that $H$ contains a clique on $[t-1]$, so $|N(x_0,t_1)| = \Omega(t^{r-2})$, and moreover every vertex in $[t-1]$ has weight $\Omega(t^{-1})$ (by definition of $H$ and \Cref{prop:prelims} \ref{itm:useful-w-t-1}). 
			
			Using \Cref{cl:weight-loss-swaps}, we have 
				$$
					w(G) - w'(H) = w(G) - w(H) = O(t^{-0.1}w(t)^2),
				$$
				 hence $w''(H) > w(G)$, a contradiction to the assumption that $G$ maximises the Lagrangian among $r$-graphs with the same number of edges. This completes the proof of \Cref{thm:clique}. 
		\end{proof}

\section{Counterexamples to the Frankl-F\"uredi conjecture}\label{sec:counter}

	In this section we find an infinite family of counterexamples to the Frankl-F\"uredi conjecture (\Cref{conj:TheConj}), for each $r \ge 4$. We achieve this by revealing a connection between the Lagrangian of an $r$-graph $G$ and \emph{the sum of degrees squared} of $G$, denoted $P_2(G)$, where for a hypergraph $H$, we define
	\[	
		\ds{H} := \sum_{x \in V(H)} d(x)^2,
	\] 
	where $d(x)$ is the degree of $x \in H$, i.e.\ the number of edges of $H$ incident with $x$. We also define
	\begin{align*}
		& P_2(r,m,t):= \max \{P_2(H): H \subseteq [t]^{(r)}, |H| = m\}.
	\end{align*}
	The problem of characterising the $r$-graphs $H$ with $t$ vertices satisfying $P_2(H) = P_2(r,|H|,t)$ is a difficult question which has been studied by a number of different people. In \cite{Us2} we give a more detailed discussion of the main results concerning $P_2(H)$ (or see, for example,~\cite{AhlKat,Olp,PelPetSte,AbrFerNeuWat,Cae,Nik07,Bey}). 
	
	The main result of this section is the following theorem, which we shall use to give counterexamples to \Cref{conj:frankl-furedi}. 

	\begin{thm} \label{thm:counter-exs}
		Let $r \ge 4$, $2 \le i \le r-2$ and $t \in \N$.
		Suppose that $G$ is a subgraph of $[t]^{(r)}$ that maximises the Lagrangian among $r$-graphs with $|G|$ edges, and suppose that $G$ is left-compressed and that every member of $\overline{G}:= [t]^{(r)} \setminus G$ contains $I:= \{t-(i-1), \ldots, t\}$. Let $H \subseteq [t-i]^{(r-i)}$ have edge set $\{e \setminus I: e \in G, I \subseteq e\}$. Then $\ds{H} = (1 + O(t^{-(i-1)}))P_2(|H|,r-i,t-i).$
	\end{thm}

	In the next subsection, we use \Cref{thm:counter-exs} to derive \Cref{thm:main2} and give some further comments regarding the statement. We then prove \Cref{thm:counter-exs} in \Cref{subsec:counter-exs}.

	\subsection{Consequences of \Cref{thm:counter-exs}} \label{subsec:remarks}

		To illustrate how \Cref{thm:counter-exs} can be used, we give a simple example.

		\begin{ex} \label{ex:4-4}
			We consider the case $r = 4$ and $m = \binom{t}{r} - \binom{t-2}{r-2} + 4$ here. 
			Let $G$ be the colex graph $\col{m}{r}$. Let $W$ be the set of quadruples in $[t]$ that do not contain $\{t-1,t\}$. So $G$ can be written as
			\begin{align*}
				& G =  W \cup \{Y \cup \{t-1,t\}:Y \in \mathcal{C}(4,2)\}.
			\end{align*}
			In particular, the graph $H$, of pairs of vertices whose union with $\{t-1, t\}$ forms an edge in $G$, is $H = \raisebox{-8pt}{\includegraphics{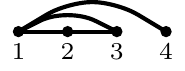}} = \col{4}{2}$.
			Consider the $r$-graph $G'$ with $m$ edges, defined as follows.
			\begin{align*}
				&G' :=  W \cup \left\{Y \cup \{1,t-1,t\}: Y \in \{2,3,4,5\}\right\}.
			\end{align*}
			In this case, the graph $H'$, defined as before with respect to $G'$, is $H' = \raisebox{-8pt}{\includegraphics{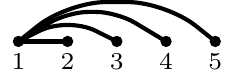}}$, which is the star on five vertices.
			One can check that $\ds{H'} = 20 > 18 = \ds{H}$. It follows from \Cref{thm:counter-exs} that $G$ does not maximise the Lagrangian, among $r$-graphs with $m$ edges.
		\end{ex}

		Note that if $G$ is a colex graph as in \Cref{thm:counter-exs}, then $H$ is a colex $(r-i)$-graph. However, as in \Cref{ex:4-4}, in many cases the colex graph does not maximise, nor is it close to maximising, the sum of degrees squared among $(r-i)$-graphs with order $t-i$ and the same size as $H$. Therefore, \Cref{thm:counter-exs} allows us to find a wide range of counterexamples to \Cref{conj:frankl-furedi}. We now apply \Cref{thm:counter-exs} to prove \Cref{thm:main2}.
		
		\begin{proof}[Proof of \Cref{thm:main2}]
			Let $G:= \col{m}{r}$, where $m = \binom{t}{r} - \binom{t-2}{r-2} + s$, and suppose that $r \le s \le \alpha_r \binom{t-2}{r-2}$, where $\alpha_r := (12(r-2))^{-(r-2)/(r-3)}$. Then, in the language of \Cref{thm:counter-exs}, we have $i = 2$ and $H \simeq \col{s}{r-2}$. By \Cref{thm:counter-exs}, if $G$ maximises the Lagrangian among $r$-graphs with $m$ edges, then $\ds{H} = (1 + O(t^{-1}))P_2(s, r-2, t-2)$, i.e.\ $H$ is close to maximising $\ds{\cdot}$ among $(r-2)$-graphs with $s$ edges and $t-2$ vertices.
			Let $H'$ be the \emph{lex graph} $\lex{s}{r-2}{t-2}$ (for general $s, r, t$ the \emph{lex} graph $\lex{s}{r}{t}$ is the $r$-graph on vertex set $[t]$ with $s$ edges whose edges form an initial segment in the lexicographic order on $[t]^{(r)}$, according to which $A <_{\lexx} B$ if and only if $\min{A \triangle B} \in A$).
			It follows from \Cref{prop:ds-lex-colex,prop:ds-lex-colex-small} below, that $\ds{H} \le (1 - \Omega(1))\ds{H'}$ (indeed, this follows from \Cref{prop:ds-lex-colex} if $\binom{8(r-2)^2}{r-2} \le s \le \alpha_r \binom{t-2}{r-2}$, and it follows from \Cref{prop:ds-lex-colex-small} if $s \ge r$ and $s = O(1)$).

			\begin{prop} \label{prop:ds-lex-colex}
				Let $\binom{8r^2}{r} \le s \le \beta_r \cdot \binom{t}{r}$, where $\beta_r := (12r)^{-r/(r-1)}$. Then $$\ds{\col{s}{r}} \le \frac{1}{2} \cdot \ds{\lex{s}{r}{t}}.$$
			\end{prop}

			\begin{prop} \label{prop:ds-lex-colex-small}
				Let $r+2 \le s \le t-r+1$.
				Then $$\ds{\col{s}{r}} \le (1 - \frac{1}{rs^2}) \ds{\lex{s}{r}{t}}.$$
			\end{prop}

			We postpone the proofs of \Cref{prop:ds-lex-colex,prop:ds-lex-colex-small} to the appendix as they are simple calculations.
		\end{proof}

		\begin{rem}
			The value $m = \binom{t}{r} - \binom{t-2}{r-2} + r$ is the smallest number of edges in the range $[\binom{t-1}{r}, \binom{t}{r}]$ for which \Cref{conj:TheConj} is false for $r = 4$. (By \Cref{thm:main2}, values of $m$ that form counterexamples are in the range $(\binom{t}{r} - \binom{t-2}{r-2}, \binom{t}{r})$, and we checked, but do not include the proof here, that colex is best when $m = \binom{t}{r} - \binom{t-2}{r-2} + a$ for $r=4$, $a \in \{1,2,3\}$.) However, for larger $r$ it is not clear whether there is a counterexample with fewer edges (as there are two $(r-2)$-graphs with $i$ edges, where $3 \le i \le r-1$, which are maximisers of $\ds{\cdot}$: a subgraph of the clique $[r-1]^{(r-2)}$ and the graph all of whose edges contain a fixed set of vertices of size $r-3$, and \Cref{thm:counter-exs} does not distinguish between the two\footnote{It is not hard to see that these are the only maximisers in this range, but we do not prove this here.}).
		\end{rem}

		\begin{rem}
			Recall (from \Cref{thm:r3}) that when $r = 3$ there are no counterexamples for large $t$. This is understandable in light of \Cref{thm:counter-exs} as we would have $i = 2$ so $H$ is a $1$-graph, i.e.\ a collection of singletons, and $1$-graphs are uniquely determined (up to relabelling) by the number of vertices and edges.
		\end{rem}

		\begin{rem}
			While the statement of \Cref{thm:counter-exs} looks quite specialised, it applies to all colex graphs. Indeed, as long as a colex graph has enough non-edges (with respect to $[t]^{(r)}$), then each of them contains the set $I = \{t-i+1, \ldots, t\}$.
			In fact, in \cite{Us2} we prove that for almost all values of $m$ with $\binom{t}{r} - \binom{t-i}{r-i} \le m \le \binom{t}{r}$, there is a maximiser of the Lagrangian which is a subgraph of $[t]^{(r)}$ whose non-edges all contain the set $\{t-i+1, \ldots, t\}$. With this in hand, \Cref{thm:counter-exs} can be used to find maximisers of the Lagrangian for many values of $m$. We elaborate more on this in \cite{Us2}.
		\end{rem}

		\begin{rem} \label{rem:counter-exs-max-exactly}
			In \Cref{thm:counter-exs}, if $e(H)\cdot t^{r-i-1} \cdot t^{-(i-1)} = o(1)$ (i.e.\ if $e(H) = o(t^{2i-r})$) then $H$ maximises (precisely, not just asymptotically) $\ds{H'}$ among $(r-i)$-graphs $H'$ with $t-i$ vertices and the same number of edges as $H$ (as $\ds{H} = O(e(H) \cdot t^{r-i-1})$).
		\end{rem}
	\subsection{Proof of \Cref{thm:counter-exs}} \label{subsec:counter-exs}

	We now turn to the proof of \Cref{thm:counter-exs}. First we give a very brief overview of the proof. We start by estimating $w(x)$ in terms of the number of edges of $H$ incident with $x$, and then we use these estimates to compare $w(G)$ and $w'(G')$, where $G'$ is defined as follows, and $w'$ is a suitable weight function. 
	Recall that $H$ is the collection of $(r-i)$-sets whose union with $I = \{t-(i-1), \ldots, t\}$ is an edge of $G$; in particular $H \subseteq [t-i]^{(r-i)}$. Let $H'$ be an $(r-i)$-graph on vertex set $[t-i]$ with $|H|$ edges, that maximises $P_2(\cdot)$ among all such graphs. We obtain $G'$ by removing from $G$ all edges that contain $I$ and adding as edges all $r$-sets that are a union of $I$ with an edge of $H'$.

		\begin{proof}[Proof of \Cref{thm:counter-exs}]
			Let $w$ be a maximal weighting of $G$; recall that $G$ is left-compressed, so $w$ is decreasing. Say that $x \not= y \in V(G)$ are \emph{twins} if $N_x(y) = N_y(x)$. Note that the vertices $t-(i-1), \ldots, t$ are twins in $G$. Thus, by \eqref{eqn:wvw}, they have the same weight, unless $i = 2$ and there is no edge that contains both $t$ and $t-1$, but in the latter case $H = \emptyset$, and the statement of \Cref{thm:counter-exs} clearly holds. We assume that the former holds, and denote the weight of the vertices $t-(i-1), \ldots, t$ by $\beta$. Write 
			\begin{align*}
				& \Delta = \frac{\beta^i \, w(1)^{r-i-1}}{w(N(1,t-i))} 
				& \alpha = w(t-i) - d(t-i) \cdot \Delta,
			\end{align*}
			where $d(x)$ denotes the degree of $x$ in $H$. For $x \in [t-i]$ write $w(x) = \alpha + \delta(x)$.

			\begin{claim} \label{cl:counter-ex-weight}
				The following estimates hold.
				\begin{enumerate}
					\item \label{itm:Delta}
						$\Delta = O(t^{-(r-1)})$.
					\item \label{itm:alpha}
						$\alpha = t^{-1} + O(t^{-i})$.
					\item \label{itm:delta-x}
						$\delta(x) = d(x) \cdot \Delta (1 + O(t^{-(i-1)}))$ for $x \in [t-i]$.
				\end{enumerate}
			\end{claim}

			\begin{proof}
				Recall that $w(x) = \Theta(t^{-1})$ for every $x \in [t-1]$ by \Cref{prop:bigO} \ref{itm:vx-bound} and \Cref{prop:prelims} \ref{itm:useful-w-t-1}. In particular, $\beta = \Theta(t^{-1})$. Also, since the number of non-edges containing $\{1, x\}$ for any $x \in [t]$ is $O(t^{r-i-1})$ (as all non-edges contain $\{t-(i-1), \ldots, t\}$), we have 
				\begin{equation} \label{eqn:w-N-1-x}
					w(N(1,x)) = \Theta(1) \,\,\,\,\,\, \text{for $x \in [t]$}.
				\end{equation}
				It follows that $\Delta = O(t^{-(r-1)})$, as required for \ref{itm:Delta}.

				Again by \eqref{eqn:w-N-1-x}, and by \Cref{cor:frankl}, $w(1) - w(x) = O\big(w(N_x(1)) - w(N_1(x))\big) = O(t^{-i})$ for $x \in [t-i]$, where the second inequality follows as the number of absent edges that contain $x$ is $O(t^{r-i-1})$, and the weight of each $(r-1)$-set is $O(t^{-(r-1)})$. Similarly, for $x \in \{t-(i-1), \ldots, t\}$ we have $w(1) - w(x) = O(t^{-(i-1)})$. It follows that
				\begin{equation} \label{eqn:w-x}
					w(x) = \left\{
						\begin{array}{ll}
							w(1) + O(t^{-i}) & x \in [t-i] \\
							w(1) + O(t^{-(i-1)}) & x \in \{t-(i-1), \ldots, t\}.
						\end{array}
					\right.
				\end{equation}
				Summing this over all $x \in [t]$ and rearranging, gives $w(1) = t^{-1} + O(t^{-i})$. Applying \eqref{eqn:w-x} again, using \ref{itm:Delta} and the bound $d(t-i) = O(t^{r-i-1})$, gives $\alpha = t^{-1} + O(t^{-i})$, as required for \ref{itm:alpha}.

				Next, we compare $w(N(x, t-i))$ and $w(N(1, t-i))$, for $x \in \{2, \ldots, t-i-1\}$. As $G$ is left-compressed, $w(N(x,t-i)) - w(N(1,t-i))$ is equal to 
				\[
					(w(1) - w(x)) \cdot w(N(1,x,t-i)) - w\big(N_x(1, t-i) \setminus N_1(x,t-i)\big) = O(t^{-i}),
				\]
				using \eqref{eqn:w-x}, the inequality $w\big(N(1,x,t-i)\big) \le 1/(r-3)!$ (by \Cref{lem:lagclique} \ref{itm:easy-ub}), and that all absent edges contain $\{t-(i-1), \ldots, t\}$. It follows from \eqref{eqn:w-N-1-x} that
				\begin{equation} \label{eqn:compare-w-N}
					w(N(x, t-i)) = w(N(1, t-i)) (1 + O(t^{-i}))\,\,\,\,\,\, \text{for $x \in [t-i-1]$}.
				\end{equation}
				Using \Cref{cor:frankl} again, we find that, for $x \in [t-i-1]$, 
				\begin{align*}
					w(x) - w(t-i) 
					& = \frac{1}{w(N(x,t-i))} \cdot \Big(w(N_{t-i}(x)) - w(N_x(t-i) )\Big) \\ 
					& = \frac{1}{w(N(1,t-i))} \cdot \big(d(x) - d(t-i)\big) \cdot \beta^i w(1)^{r-i-1} \cdot \left(1 + O(t^{-(i-1)})\right) \\
					& = (d(x) - d(t-i)) \cdot \Delta \cdot (1 + O(t^{-(i-1)})) \\
					& = (d(x) - d(t-i)) \cdot \Delta + O\big(d(x) \cdot \Delta \cdot t^{-(i-1)}\big).
				\end{align*}
				The second equality holds due to \eqref{eqn:compare-w-N}, because $N_x(t-i) \subseteq N_{t-i}(x)$, and because the weight of each edge in $N_{t-i}(x)$ is $\beta^i w(1)^{r-i-1} (1 + O(t^{-(i-1)}))$ (using \eqref{eqn:w-x}); the last equality holds as $d(x) \ge d(t-i)$. It follows that, for $x \in [t-i]$, 
				\begin{align*}
					w(x) &= w(t-i) - d(t-i) \cdot \Delta + d(x) \cdot \Delta (1 + O(t^{-(i-1)}))\\
					 &= \alpha + d(x) \cdot \Delta (1 + O(t^{-(i-1)})), 
				\end{align*}
				as required for \ref{itm:delta-x}. 
			\end{proof}

			Let $G'$ be another $r$-graph on vertex set $[t]$ for which every member of $\overline{G'}$ contains $\{t-(i-1), \ldots, t\}$ and suppose that $G'$ has the same number of edges as $G$. Define $H'$ analogously to $H$, and denote by $d'(x)$ the degree of a vertex $x$ in $H'$. By assumption, $\lambda(G) \ge \lambda(G')$. Our aim is to show that $\ds{H} \ge \ds{H'}(1 + O(t^{-(i-1)}))$. Denote the number of edges of $H$ by $m$; then $H'$ also has $m$ edges.

			In order to compare $\lambda(G)$ with $\lambda(G')$ we will define a modified weight function $w'$ such that
			\begin{equation} \label{eqn:defn-w'}
				w'(x) = 
					\left\{ 
						\begin{array}{ll}
							\beta & x \in \{t - (i-1),\ldots,t\} \\
							\alpha + d'(x) \cdot \Delta (1 + O(t^{-(i-1)})) & \text{ otherwise.}
						\end{array}
					\right.
			\end{equation}
			For $x \in [t-i]$, write 
			\[
				w(x) = \alpha + d(x) \cdot \Delta (1 + \eps(x)).
			\] 
			Then, by \Cref{cl:counter-ex-weight}, we have $\eps(x) = O(t^{-(i-1)})$. Let 
			\[
				\zeta := \frac{d(1)\eps(1) + \ldots + d(t-i)\eps(t-i)}{d(1) + \ldots + d(t-i)},
			\]  
			and note that $\zeta = O(t^{-(i-1)})$, as $\zeta$ is the weighted average of $\eps(x)$ with $x \in [t-i]$, and $\eps(x) = O(t^{-(i-1)})$ for every $x \in [t-i]$. For $x \in [t-i]$, set 
			\[
				w'(x) = \alpha + \delta'(x) \text{ and } \delta'(x) = d'(x) \cdot \Delta \cdot (1 + \zeta).
			\]
			Observe that $\sum_{x \in [t-i]} \delta'(x) = \sum_{x \in [t-i]} \delta(x)$, from which it follows that $w'$ is a legal weight function that satisfies \eqref{eqn:defn-w'}.
			
			In the next two claims, \Cref{cl:weight-diff-H-edges,cl:weight-diff-main-edges}, we compare $w(G)$ with $w'(G')$.
			
			\begin{claim} \label{cl:weight-diff-H-edges}
				The difference between the weight, with respect to $w'$, of edges in $G'$ that contain $\{t-(i-1), \ldots, t\}$ and the weight, with respect to $w$, of edges in $G$ that contain $\{t-(i-1), \ldots, t\}$ is 
				\begin{align*}
					\frac{\Delta^2}{(r-2)!} \cdot \Big(\ds{H'} - \ds{H} + O\!\left(t^{-(i-1)} (\ds{H} + \ds{H'}) \right) \Big)\big(1 + O(t^{-1})\big).
				\end{align*}
			\end{claim}

			\begin{proof}
				Note that the required quantity is $\beta^i \cdot (w'(H') - w(H))$. Let us evaluate $w(H)$.
				\begin{align*}
					w(H) 
					& = \sum_{(x_1, \ldots, x_{r-i}) \in E(H)} w(x_1) \cdots w(x_{r-i}) \\
					& = \sum_{(x_1, \ldots, x_{r-i}) \in E(H)} (\alpha + \delta(x_1)) \cdots (\alpha + \delta(x_{r-i})) \\
					& = \sum_{0 \le j \le r-i} \alpha^{r-i-j} \sum_{1 \le x_1 \le \ldots \le x_j \le t-i} \delta(x_1) \cdots \delta(x_j) \cdot d(x_1, \ldots, x_j) \\
					& = m \alpha^{r-i} + \alpha^{r-i-1} \cdot \Delta \cdot \ds{H} \left(1 + O(t^{-(i-1)})\right), 
				\end{align*}
				where $d(x_1, \ldots, x_j)$ is the number of edges in $H$ that contain $\{x_1, \ldots, x_j\}$. Indeed, we used the facts that $\delta(x) = d(x) \cdot \Delta (1 + O(t^{-(i-1)})) = O(t^{-i})$ (by \Cref{cl:counter-ex-weight} \ref{itm:delta-x} and \ref{itm:Delta} and by $d(x) = O(t^{r-i-1})$) and $\sum_{x_2, \ldots, x_j} d(x_1, \ldots, x_j) = O(d(x_1))$ to obtain the following bound for $j \ge 2$.
				\begin{align*}
					 \alpha^{r - i - j} \sum_{x_1, \ldots, x_j} \delta(x_1) \cdots \delta(x_j) \cdot d(x_1, \ldots, x_j)  
					& =  O(\alpha^{r-i-1} \cdot t^{j-1} \cdot \sum_{x \in [t-i]} d(x)^2 \Delta \cdot t^{(j-1)i}) \\
					& =  O\!\left(\alpha^{r-i-1} \cdot \Delta \cdot \ds{H} \cdot t^{-(i-1)}\right).
				\end{align*}
				A similar argument shows that 
				\begin{align*}
					w'(H')  
					& =  m \alpha^{r-i} + \alpha^{r-i-1} \Delta \cdot \ds{H'} \left(1 + O(t^{-(i-1)})\right). 
				\end{align*}
				Hence 
				\begin{align*}
					& \beta^i\cdot \big(w'(H') - w(H)\big) \\ 
					= \,& \beta^i \cdot \alpha^{r-i-1} \cdot  \Delta \cdot \left(\ds{H'} - \ds{H} + O\!\left(t^{-(i-1)} \left(\ds{H} + \ds{H'} \right) \right) \right) \\
					= \, & \frac{\Delta^2}{(r-2)!} \cdot \Big(\ds{H'} - \ds{H} + O\!\left(t^{-(i-1)} \left(\ds{H} + \ds{H'} \right) \right) \Big)\big(1 + O(t^{-1})\big),
				\end{align*}
				where we used the  definition of $\Delta$ and the estimates $\alpha = w(1)\big(1 + O(t^{-(i-1)})\big)$ and $w(N(1, t-i)) = \frac{1}{(r-2)!} + O(t^{-1})$.
			\end{proof}

			Next we evaluate the contribution of edges of $G$ that do not contain $\{t-(i-1), \ldots, t\}$. 
			\begin{claim} \label{cl:weight-diff-main-edges}
				The difference between the weight of $r$-tuples that do not contain $\{t-(i-1), \ldots, t\}$ with respect to $w$ and $w'$ is
				\begin{align*}
					\frac{\Delta^2}{2(r-2)!} \cdot \Big(\ds{H'} - \ds{H} + O\!\left(t^{-(i-1)} \left(\ds{H} + \ds{H'} \right) \right) \Big)\big(1 + O(t^{-1})\big).	
				\end{align*}
			\end{claim}

			\begin{proof}
				Given $0 \le l \le i-1$, the difference between the weight of $r$-tuples that contain exactly $l$ vertices from $\{t-(i-1), \ldots, t\}$ in $w$ and in $w'$ is the following times $\binom{i}{l} \beta^l$.
				\begin{align*}
					& \sum_{x_1 < \ldots < x_{r-l}} 
					\Big( w(x_1) \cdots w(x_{r-l})  -  w'(x_1) \cdots w'(x_{r-l})\Big)\\
					= & \sum_{x_1 < \ldots < x_{r-l}} 
					\Big( (\alpha + \delta(x_1)) \cdots (\alpha + \delta(x_{r-l})) - (\alpha + \delta'(x_1)) \cdots (\alpha + \delta'(x_{r-l})) \Big) \\
					= & \sum_{0 \le j \le r-l} \alpha^{r-l-j} \binom{t-i-j}{r-l-j} 
					\sum_{x_1 < \ldots < x_j} \Big( \delta(x_1) \cdots \delta(x_j) - \delta'(x_1) \cdots \delta'(x_j) \Big) \\
					= & \, \alpha^{r-l-2} \binom{t-i-2}{r-l-2} \cdot \frac{1}{2} \left(
					\sum_x \delta(x)^2 - \sum_x \delta'(x)^2 + O\Big(t^{-(i-1)} \sum_x \big(\delta(x)^2 + \delta'(x)^2\big) \Big) 
					\right) \\
					= & \, 	\frac{\Delta^2}{2(r-l-2)!} \cdot \Big(\ds{H'} - \ds{H} + O\!\left(t^{-(i-1)} \left(\ds{H} + \ds{H'} \right)\right) \Big)\big(1 + O(t^{-1})\big).
				\end{align*}
				Indeed, for the penultimate equality we used the fact that $\sum_{x \in [t-i]} \delta(x) = \sum_{x \in [t-i]} \delta'(x)$ which implies that the summands with $j=0$ and $j=1$ are $0$. Furthermore, for $j = 2$, we used the equation $\sum_{x < y} \delta(x) \delta(y) = \frac{1}{2} \left( \big(\sum_x \delta(x) \big)^2 - \sum_x \delta(x)^2\right)$, which implies that 
				\[
					\sum_{x < y} \delta(x)\delta(y) - \sum_{x < y} \delta'(x) \delta'(y) = \frac{1}{2} \left(\sum_x \delta(x)^2 - \sum_x \delta'(x)^2\right).
				\]
				Similarly, for $j \ge 3$ we have 
				\begin{align*}
					\sum_{x_1 < \ldots < x_j} \delta(x_1) \cdots \delta(x_j) 
					& = \frac{1}{j!} \big(\sum_x \delta(x) \big)^j + O\big(\sum_x \delta(x)^2 \cdot \sum_{x_1 < \ldots < x_{j-2}} \delta(x_1) \cdots \delta(x_{j-1}) \big) \\
					& = \frac{1}{j!} \big(\sum_x \delta(x) \big)^j + O\big(\sum_x \delta(x)^2 \cdot t^{j-2} \cdot t^{-(j-2)i}\big),
				\end{align*}
				using $\delta(x) = O(t^{-i})$.
				We conclude that 
				\[ 
					\sum \delta(x_1) \cdots \delta(x_j) - \sum \delta'(x_1) \cdots \delta'(x_j) = O\bigg(t^{-(i-1)} \Big(\sum_x \delta(x)^2 + \sum_x \delta'(x)^2\Big) \bigg).
				\]
				For the last equality above we used the fact that $\delta(x) = d(x) \cdot \Delta(1 + O(t^{-(i-1)}))$ and similarly for $\delta'(x)$; also, we used the fact that $\alpha = w(1) + O(t^{-i})$.

				\Cref{cl:weight-diff-main-edges} follows: for $1 \le l \le i-1$, the contribution of the edges with exactly $l$ vertices from $\{t-(i-1), \ldots, t\}$ is accounted for in the $O(t^{-1})$ error term, and the main term accounts for the edges with no vertices in $\{t-(i-1), \ldots, t\}$ (unsurprisingly, as there are much more of the latter type of edges than the former).
			\end{proof}

			By \Cref{cl:weight-diff-H-edges,cl:weight-diff-main-edges}, we have  that $w'(G') - w(G)$ is equal to 
			\begin{align*}
						 \frac{\Delta}{2(r-2)!} \left( \ds{H'} - \ds{H} + O(t^{-(i-1)}\left(\ds{H} + \ds{H'} \right))\right)\left(1 + O(t^{-1})\right). 
			\end{align*}
			So since $\lambda(G) = \Lambda(|G|,r)$, we have $\ds{H} \ge \ds{H'}(1 + O(t^{-(i-1)}))$. Since $H'$ can be chosen arbitrarily, by taking $H'$ such that $\ds{H'} = P_2(m,r-i,t-i)$, we find that $\ds{H} = (1 + O(t^{-(i-1)}))P_2(m,r-i,t-i)$, as required for the proof of \Cref{thm:counter-exs}. 
		\end{proof}
\section{Conclusion} \label{sec:conclusion}
	In this paper we showed that the Lagrangian of an $r$-graph with a given number of edges $m$ is not always maximised by an initial segment of colex of size $m$ (see \Cref{thm:main2,thm:counter-exs}), thus disproving a longstanding conjecture of Frankl and F\"uredi. We do however show that for almost all values of $m$, the colex graph is a maximiser (see \Cref{thm:main1}), and we obtain some understanding of the structure of maximisers for other values of $m$ (see \Cref{thm:counter-exs}). It would of course be interesting to find maximisers of the Lagrangian for all $m$, but we suspect that this is a very hard problem. An indication to the difficulty of this problem is the relation to the problem of maximising the sum of degrees squared, over $r$-graphs with given order and size (explored in \Cref{sec:counter}), which in itself appears very hard.

	\subsection*{Acknowledgements}

		This research was partially completed while the third author was visiting ETH Zurich. The third author would like to thank Benny Sudakov and the London Mathematical Society for making this visit possible. 
		We would also like to thank Rob Morris and Imre Leader for their helpful comments and advice.

	\bibliography{lagra}
    \bibliographystyle{amsplain}

	\appendix
	\section{Proof of \Cref{prop:ds-lex-colex,prop:ds-lex-colex-small}}

		\begin{proof}[Proof of \Cref{prop:ds-lex-colex}]
			Let $\LL := \lex{s}{r}{t}$ and $\C := \col{s}{r}$. 
			Let $k$ be the integer satisfying $\binom{k-1}{r} < s \le \binom{k}{r}$. So, since $s \ge \binom{8r^2}{r}$, we have $k \ge 8r^2$.
			Also,
			\begin{equation} \label{eqn:ds-colex}
				\ds{\C} \le \ds{[k]^{(r)}} = k \cdot \binom{k-1}{r-1}^2.
			\end{equation}

			Suppose first that $s \le \binom{t-1}{r-1}$. Then all edges in $\LL$ contain the vertex $1$, i.e.\ $1$ has degree $s$, so $\ds{\LL} \ge s^2$.
			As $\binom{k-1}{r-1} = \frac{r}{k-r} \cdot \binom{k-1}{r} \le \frac{2rs}{k}$ (using $k \ge 2r$ for the last inequality), it follows from \eqref{eqn:ds-colex} that 
			\[
				\ds{\C} \le \frac{4r^2s^2}{k} \le \frac{s^2}{2} \le \frac{1}{2} \cdot \ds{\LL},
			\]
			as required.

			Next, we assume that $s \ge \binom{t-1}{r-1}$. Then, by definition of $\LL$, there are at least $\floor{s / \binom{t-1}{r-1}} \ge s / (2\binom{t-1}{r-1})$ vertices in $\LL$ of degree $\binom{t-1}{r-1}$ (where the inequality follows from the lower bound on $s$). Thus,
			\[
				\ds{\LL} \ge \frac{s}{2 \binom{t-1}{r-1}} \cdot \binom{t-1}{r-1}^2 = \frac{s}{2} \cdot \binom{t-1}{r-1}.
			\]
			On the other hand, using $\binom{k-1}{r-1} \le \frac{2rs}{k}$ again,
			\[
				\ds{\C} \le 2rs \cdot \binom{k-1}{r-1}.
			\]
			Thus, in order to obtain the desired result, it suffices to show that 
			\[
				\binom{k-1}{r-1} \le \frac{1}{8r} \binom{t-1}{r-1}.
			\]
			To show this, we first obtain an upper bound on $k$, which holds for sufficiently large $t$. 
			\begin{align*}
				s 
				&\le (12r)^{-r/(r-1)} \cdot \binom{t}{r} 
				\le (12r)^{-r/(r-1)} \cdot \frac{t^r}{r!} \\
				&= \frac{ \left( (12r)^{-1/(r-1)} \cdot t \right)^r}{r!} 
				\le \frac{ \left( (10r)^{-1/(r-1)} \cdot t - r \right)^r}{r!} 
				\le \binom{(10r)^{-1/(r-1)} \cdot t}{r}.	
			\end{align*}
			It follows that $k \le (10r)^{-1/(r-1)} \cdot t$, which implies that
			\[
				\binom{k-1}{r-1} 
				\le \frac{k^{r-1}}{(r-1)!}
				\le \frac{1}{10r} \cdot \frac{ t^{r-1}}{(r-1)!}
				\le \frac{1}{8r} \cdot \frac{(t-r)^{r-1}}{(r-1)!}
				\le \frac{1}{8r} \cdot \binom{t-1}{r-1},
			\]
			as desired.
		\end{proof}

		\begin{proof}[Proof of \Cref{prop:ds-lex-colex-small}]
			Let $\LL := \lex{s}{r}{t}$ and $\C := \col{s}{r}$.
			As $s \le t - r + 1$, all edges in the graph $\LL$ contain $[r-1]$. In particular, the vertices $1, \ldots, r-1$ have degree $s$, there are $s$ vertices of degree $1$, and the remaining vertices are isolated. Thus
			\[
				\ds{\LL} = (r-1)s^2 + s.
			\]
			We now give a very crude upper bound on $\ds{\C}$. Note that $\C$ contains both edges $(1, \ldots, r-1, r+2)$ and $(1, 3, \ldots, r+1)$, whose intersection has size $r-2$. It follows that
			\begin{align*}
				\ds{\C}
				&= \sum_{e \in E(\C)} \sum_{x \in e} d(x)
				= \sum_{e,f \in E(\C)} |e \cap f| 
				\le sr + \sum_{e, f \in E(\C), \,\, e \neq f} |e \cap f|
				< sr + s(s-1)(r-1)  = \ds{\LL},
			\end{align*}
			where the inequality follows as $|e \cap f| \le (r-1)$ for every distinct edges $e$ and $f$, and there is a strict inequality for at least one pair of distinct edges.
			It follows that
			\[
				\ds{\C} \le \ds{\LL} - 1 = \left(1 - 1/\ds{\LL} \right) \ds{\LL} \le \left(1 - 1/rs^2\right) \ds{\LL},
			\]
			as required.
		\end{proof}

\end{document}